\documentclass{article}
\usepackage{fullpage}
\usepackage{amsmath}
\usepackage{amsfonts}
\usepackage{amssymb}
\usepackage{epsf}
\usepackage{graphics}
\usepackage{float}
 \usepackage{graphicx}
 \usepackage{color}

\usepackage{mathrsfs} 
\usepackage{breakcites}
\usepackage{breqn}

\usepackage{fullpage}
\usepackage{amsmath}
\usepackage{amsfonts}
\usepackage{amssymb}
\usepackage{epsf}
\usepackage{graphics}
\usepackage{float}
\usepackage{graphicx}
\usepackage{color}
\usepackage{ulem}
\usepackage{breqn}

\restylefloat{figure}
\usepackage{amsthm}
\usepackage{rotating}
\usepackage{subfigure}   
\usepackage{verbatim}
\usepackage[printonlyused]{acronym} 
\usepackage{setspace}    
\definecolor{DarkGreen}{rgb}{0.1,0.6,0.0}

\newcommand{\vect}[1]{{\bf {#1} }}

\usepackage{multimedia}

\theoremstyle{plain}
\newtheorem{theorem}{Theorem}
\newtheorem{proposition}[theorem]{Proposition}

\newcommand{\ba}{\begin{array}}
\newcommand{\ea}{\end{array}}
\newcommand{\beq}{\begin{equation}}
\newcommand{\eeq}{\end{equation}}
\newcommand{\bea}{\begin{eqnarray}}
\newcommand{\eea}{\end{eqnarray}}
\newcommand{\bean}{\begin{eqnarray*}}
\newcommand{\eean}{\end{eqnarray*}}
\newcommand{\bal}{\begin{align}}
\newcommand{\eal}{\end{align}}
\newcommand{\bit}{\begin{itemize}}
\newcommand{\eit}{\end{itemize}}
\newcommand{\benum}{\begin{enumerate}}
\newcommand{\eenum}{\end{enumerate}}
\newcommand{\bdm}{\begin{displaymath}}
\newcommand{\edm}{\end{displaymath}}
\newcommand{\no}{\nonumber}

\newcommand{\bssz}{\begin{scriptsize}}
\newcommand{\essz}{\end{scriptsize}}
\newcommand{\bfnsz}{\begin{footnotesize}}
\newcommand{\efnsz}{\end{footnotesize}}

\newcommand{\bv}{\vect{\overline{v}}}

\newcommand{\half}{\frac{1}{2}}

\begin{document}

\bibliographystyle{apalike}    

\title{A Central-Upwind Scheme for Two-layer Shallow-water Flows with Friction and Entrainment along Channels}


\author{Gerardo Hernandez-Duenas\footnote{Gerardo Hernandez-Duenas, Institute of Mathematics, National University of Mexico, Blvd. Juriquilla 3001, Queretaro, Mexico, hernandez@im.unam.mx} \footnote{Corresponding Author} ,
Jorge Balb\'as \footnote{ Jorge Balb\'as, Department of Mathematics. California State University, Northridge. 18111 Nordhoff St. Northridge, CA 91330-8313, jorge.balbas@csun.edu}
}
 
\maketitle

\thanks{Research supported in part by grants UNAM-DGAPA-PAPIIT  IN113019 \& Conacyt A1-S-17634.}\\

\thispagestyle{empty}

\begin{abstract}
We present a new high-resolution, non-oscillatory semi-discrete central-upwind scheme for one-dimensional two-layer shallow-water flows with friction and entrainment along channels with arbitrary cross sections and bottom topography. These flows are described by a conditionally hyperbolic balance law  with non-conservative products. A detailed description of the properties of the model is provided, including entropy inequalities and asymptotic approximations of the eigenvalues of the corresponding coefficient matrix. The scheme extends existing central-upwind semi-discrete numerical methods for hyperbolic conservation and balance laws and it satisfies two properties crucial for the accurate simulation of shallow-water flows: it {\it preserves the positivity} of the water depth for each layer, and it is {\it well balanced}, {\it i.e.}, the source terms arising from the geometry of the channel are discretized so as to balance the non-linear hyperbolic flux gradients. Along with the description of the scheme and proofs of these two properties, we present several numerical experiments that demonstrate the robustness of the numerical algorithm.
\end{abstract}


\clearpage
\section*{Introduction}\label{sec:intro}

In this paper we present a new high-order numerical scheme for simulating two-layer shallow-water flows along channels with a bottom topography and varying width (see Figure \ref{fig:channels}). These flows are characterized by a large horizontal length scale relative to their depth and are commonly observed in nature --{\it e.g.}, channel flows, straits, mountain passes; and their modeling and simulation have applications in flood control, coastal engineering, or environmental assessment. Applications of this model include the study of internal waves observed in rivers and possibly caused by shear flow instabilities \cite{apel1975observations}. Shallow-water flows are typically modeled by the Saint-Venant equations, a hyperbolic balance law that results from the {\it depth averaging} of the Euler equations. The model poses various mathematical and computational challenges: the effects of the channel geometry on the flow are described by source terms that need to be discretized consistently with the existence of equilibrium solutions and the onset and propagation of discontinuities --{\it e.g.}, hydraulic jumps. The dynamics of the interface between the two layers render non-conservative products of the flow variables and their space derivatives whose correct discretization is crucial for determining the location and propagation of these discontinuities.

\begin{figure}[h!]
\centering
\hspace{-0.0in} \includegraphics[scale=0.49]{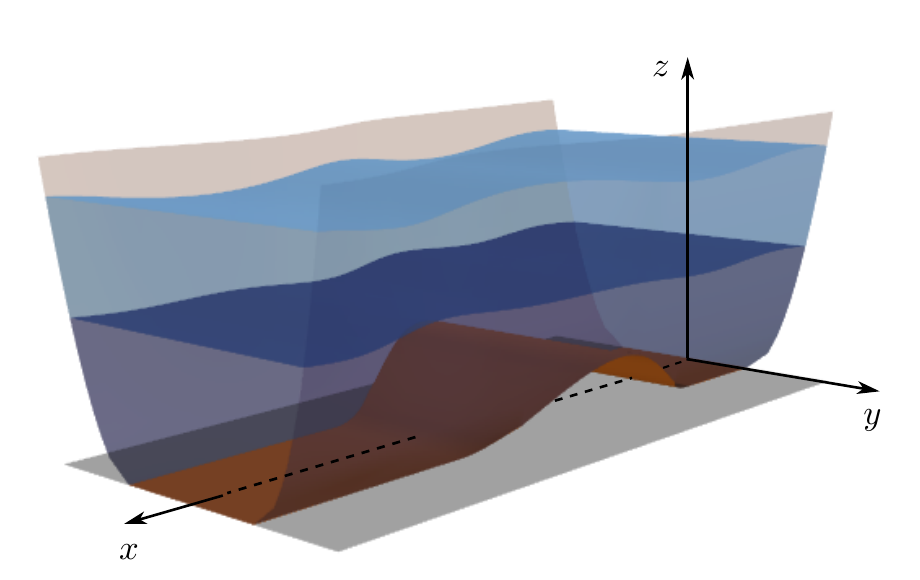} \hspace{0.82in} \includegraphics[scale=0.49]{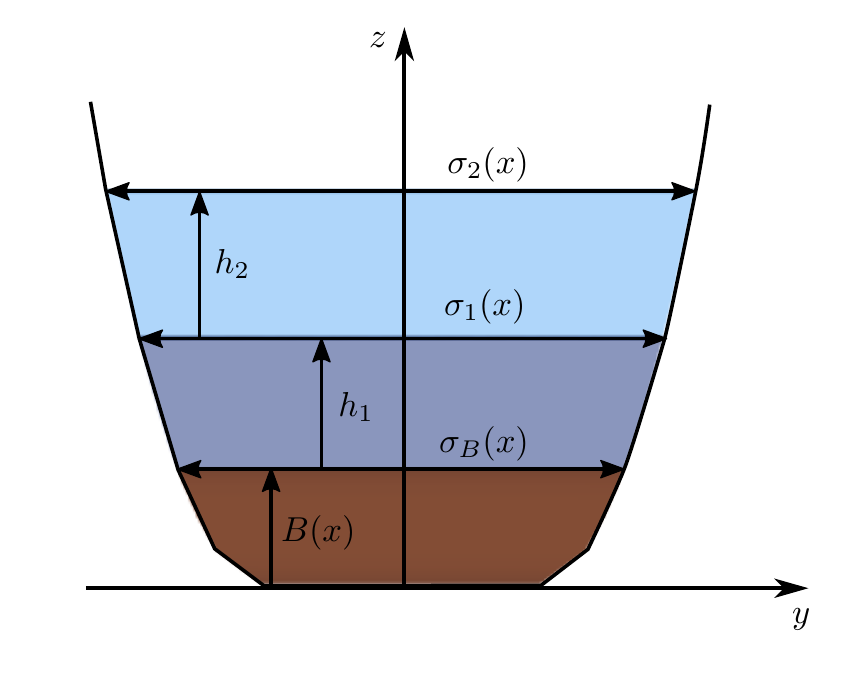} \\ \hspace{0.6in} (a) \hspace{3.2in} (b) \ \ \ \ \ \ \ \ \ \\[0.15in] 
\hspace{0.15in} \includegraphics[scale=0.49]{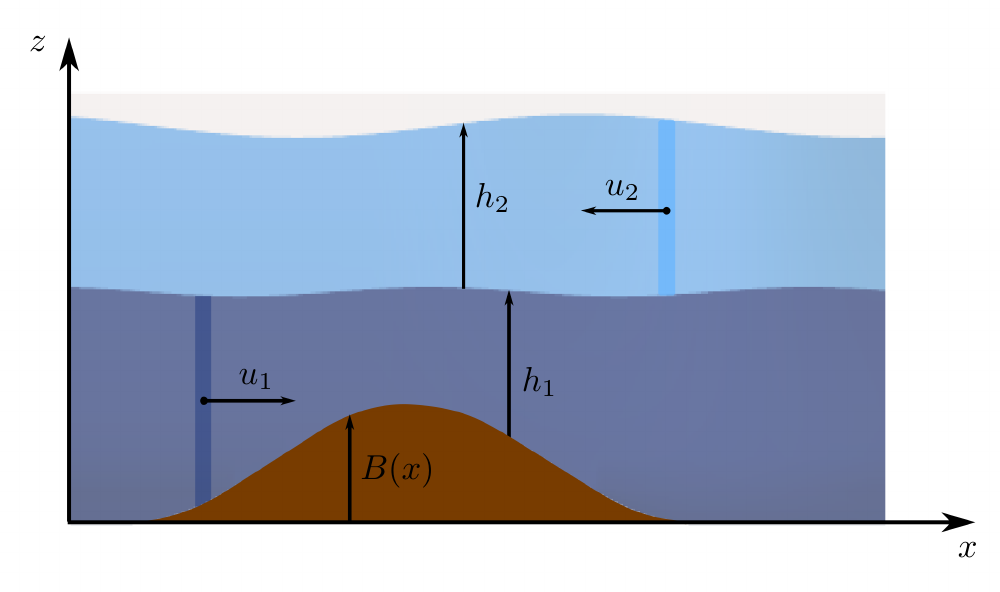} \hspace{0.5in} \includegraphics[scale=0.49]{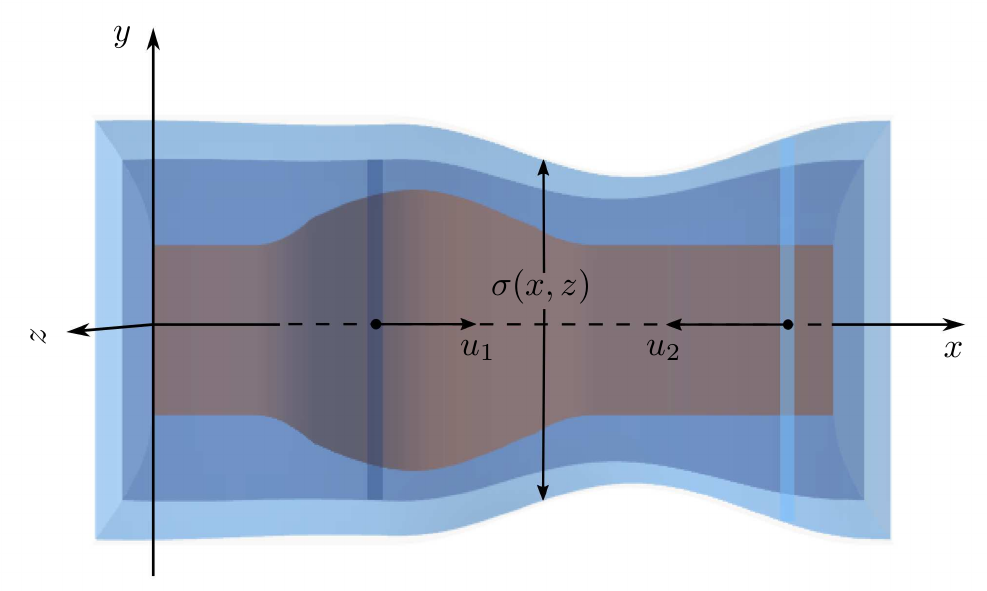} \\ \hspace{0.85in} (c) \hspace{3.15in} (d) \ \ \ \ \ \ \ \ \ \ \ \
\caption{\label{fig:channels} Schematic of a channel flow with the two-layers moving in opposite directions: (a) Full 3D view of the flow, (b) channel cross section, (c) profile view, and (d) overview of the flow.} 
\end{figure}

The various challenges that the Saint-Venant equations pose for simulating shallow-water flows are well known and have been studied extensively. The non-conservative terms may lead to the loss of hyperbolicity, and the common occurrence of steady states in geophysical flows require {\it well-balance} numerical schemes capable of capturing and resolving accurately steady-state solutions of the PDE model. These challenges become increasingly difficult to address as the flows and the corresponding PDE models that describe them increase in complexity, and a significant effort over the past couple of decades have lead to the development of numerical schemes for simulating a wide variety of flows. Numerous schemes have been proposed for the simplest case of one-layer flows along channels with constant width: a positivity preserving {\it kinetic scheme} capable of preserving the steady state of rest is presented in \cite{PerthameSimeoni2001}, and in \cite{Audusse2004} a finite volume scheme with similar properties is devised using {\it hydrostatic reconstruction} to capture steady-state solutions. The authors of \cite{khan2014modeling} proposed a {\it discontinuous Galerkin} method. And \cite{KurganovLevy2002, kurganov2007second, bollermann2011well} introduce positivity preserving well-balance {\it central-upwind} schemes for these flows. The approach suggested to achieve well-balance and positivity within the central-upwind framework is extended in \cite{BalbasKarni2009} to simulate flows along channels with {\it varying cross-sections} using a {\it central} scheme, a type of flows also addressed in \cite{Vasquez00} using an upwind Roe-type scheme. And the authors of this paper proposed a new central-upwind scheme for one-layer flows along channels with arbitrary geometry in \cite{balbas2014positivity}. Several schemes have also been devised for simulating two-layer flows. For flows along straight channels, the authors of \cite{AbgrallKarni2009} propose to address the conditional hyperbolicity of the model with a {\it relaxation} approach that simplifies greatly its eigen structure. A similar approach is presented in \cite{CHIAPOLINO20181043} where a strictly hyperbolic formulation with pressure relaxation is obtained by considering weak compressibility of the phases of the flow. And in \cite{castro2011numerical} friction between the layers is added to prevent the loss of hyperbolicity. The central-upwind scheme in \cite{KurganovPetrova2009} provides a robust approach for flows along straight channels that is simple to implement, and the work in \cite{balbas2013non} extends the central framework to channels with rectangular sections of varying width. For flows along channels with arbitrary geometry, the authors of \cite{Castro2004} extended with great success (and high impact in the field) the $Q$-scheme for hyperbolic systems with source terms previously introduced in \cite{Castro2001}. Some of these schemes for shallow-water flows have been employed --and new ones created-- for simulating, studying or recreating geophysical flows from the real world like gravitational currents or strait flows. These typically pose additional challenges such as the handling of non-symmetrical channel geometries (described by bathymetry data) or accounting for phenomena like entrainment. For some examples, we refer the reader to the works presented in \cite{adduce2012gravity, apel1975observations, Castro2004, castro2007improved, kim2008two}.

In this paper we follow on our previous works on central schemes for one-layer flows along channels with arbitrary geometry, \cite{balbas2014positivity}, and two-layer flows along channels with varying rectangular cross-sections, \cite{balbas2013non}. We propose a new high-order central-upwind scheme to compute two-layer shallow-water flows along channels with arbitrary geometry that incorporates the treatment of friction and entrainment terms. These new terms allow us to simulate more realistic flows and to assess the limitations of the model and the scheme. In order to understand and address the challenges posed by the model and its limitations, we present a detailed analysis of the hyperbolic PDE model, with special emphasis on the conditions that lead to the loss of hyperbolicity. To this end, we derive rigorous asymptotic approximations of the eigenvalues of the quasilinear form of the model and, for the sake of completeness in the analysis, we prove the existence of an entropy function and an entropy inequality that physically relevant weak solutions must satisfy.

In order to address the various numerical and computational challenges we propose a central-upwind scheme based on the semi-discrete central-upwind schemes for hyperbolic conservation laws of Kurganov, Noelle, and Petrova, \cite{kurganov2001semidiscrete}, characterized by their simple implementation and robustness. The proposed numerical scheme evolves the cell averages of the flow variables with second order accuracy, and their implementation requires four main ingredients: a {\it non-oscillatory reconstruction} of point values from cell averages that preserves the {\it positivity} of the water depth, an {\it evolution routine} to advance the solution in time, {\it estimates} of the largest and smallest eigenvalues of the system, and the {\it discretization of source terms and non-conservative products} that balance the hyperbolic fluxes so as to recognize steady states at rest and add accuracy to the computation of flows near non stationary steady states.

The paper is structured as follows. In \S \ref{sec:model} we present the model, its properties, and the challenges that these properties pose for computing numerical solutions. In \S \ref{sec:scheme} we describe the proposed numerical scheme and prove that it preserves the positivity of the water height, and it is well-balanced, {\it i.e.}, it recognizes and preserves the steady state of rest. Numerical solutions for a variety of flow regimes are presented in \S \ref{sec:Results}, validating the scheme's accuracy and robustness and demonstrating its ability to simulate a wide range of flows.
\section{The Model and its Properties} \label{sec:model}
 
The model for the two-layer shallow-water flows in channels is taken from \cite{Castro2004}. After adding source terms due to friction and entrainment, it has been re-written as
\begin{subequations} \label{eq:sw}
\begin{align}
\frac{\partial A_1}{\partial t}  & + \frac{\partial Q_1}{\partial x}   = S_e,   \label{eq:sw-1} \\[0.1in]
\frac{\partial Q_1}{\partial t}  & + \frac{\partial}{\partial x} \left(A_1u_1^2 + p_1 \right) = g \left(I_{1} - h_1 \sigma_B B'  +r h_2 \frac{\partial A_1}{\partial x} \right)+ S_{f,1} + S_e u_1, \label{eq:sw-2}  \\[0.1in]
\frac{\partial A_2}{\partial t}  & + \frac{\partial Q_2}{\partial x}   =  -r S_e,  \label{eq:sw-3} \\[0.1in]
\frac{\partial Q_2}{\partial t}  & + \frac{\partial}{\partial x} \left(A_2u_2^2+p_2 \right)  = g \left(I_{2} - h_2 \sigma_1 \frac{\partial w_1} {\partial x} \right)+ S_{f,2} -r S_e u_2.  \label{eq:sw-4} 
\end{align}
\end{subequations}

Here $h_1$, and $h_2$ denote the depth of bottom and top layers respectively, $u_1$, and $u_2$ the cross-sectional velocities; $g$ the acceleration of gravity, $B(x)$ describes the bottom topography, and $\sigma(x,z)$ the width of the channel; $A_1=  \displaystyle \int_{B}^{w_1}\sigma(x,z)\, dz$, and $A_2 = \displaystyle \int_{w_1}^{w_2}\sigma(x,z)\, dz$ are the cross-sectional {\it wet} areas in each layer; $w_1 = B + h_1$ denotes the total elevation of the internal layer and $w_2 = B+h_1+h_2$ that of the external layer; and $Q_1 = A_1 u_1$, and $Q_2 = A_2 u_2$ are the flow rates or discharges for the internal and external layers respectively. Furthermore, $\sigma_B(x) = \sigma(x,B(x)), \sigma_1(x,t)=\sigma(x,w_1(x,t)), \sigma_2(x,t) = \sigma(x,w_2(x,t))$ denote the channel's width at the bottom topography, and at the internal and external layers respectively. We note that $\sigma_1$ and $\sigma_2$ depend both on space and time since they are evaluated at the internal and external layers. The ratio of densities is denoted by $r=\rho_2/\rho_1\le 1$.  The vertically integrated hydrostatic pressure of the upper layer is given by
\begin{equation} \label{eq:p1}
p_2 =  g \int_{w_1}^{w_2} (w_2 -z) \, \sigma(x,z) \, dz,
\end{equation}
and treats the internal layer as a moving topography. The hydrostatic vertically integrated pressure in the internal layer has to account for the contribution of the pressure exerted by the upper layer on it, and it is given by
\begin{equation} \label{eq:p2}
p_1 =  g \int_{B}^{w_1} (w_1 + r h_2 - z) \, \sigma(x,z) \, dz.
\end{equation}
The source terms $I_{1}$, and $I_{2}$ correspond to the vertically integrated pressure terms due to width variation, and are given by
\beq \label{eq:IntegralTerms}
I_1 = I_1(x,t) =  \int_{B}^{w_1} (w_1 - z) \, \sigma_x (x,z) \, dz,  \ \ \ \ \ \ \text{and} \ \ \ \ \ \ I_2 = I_2(x,t)  =  \int_{w_1}^{w_2}(w_2 - z) \, \sigma_x (x,z) \, dz.
\eeq
For piecewise trapezoidal channels, the integrals involved in the computation of the cross sectional areas coincide with the corresponding approximations given by the trapezoidal rule. 

The friction terms are calculated as
\begin{equation}
\label{eq:Friction}
\begin{array}{lcl}
S_{f,1} & = & - rg\frac{n_i^2 \left| \frac{Q_1 A_1+Q_2 A_2}{A_1+A_2} \right| }{R^{4/3}} (u_1-u_2) -  g\frac{n_b^2  \left| \frac{Q_1 A_1+Q_2 A_2}{A_1+A_2} \right|}{R^{4/3}} u_1 , \\[0.2in]
S_{f,2} & = & - g\frac{n_i^2  \left| \frac{Q_1 A_1+Q_2 A_2}{A_1+A_2} \right|}{R^{4/3}} (u_2-u_1),
\end{array}
\end{equation}
for the internal and external layers respectively. Here $n_i$ and $n_b$ are the Manning roughness coefficients for the interface and bottom respectively. The hydraulic radius, $R$, is defined as the ratio between the wet area and the wetted perimeter, and it is given by
\begin{equation}
\label{eq:HydraulicRadius}
R = \frac{A_1+A_2}{\sigma_B+\int_B^{w_2} \sqrt{4+(\partial_z \sigma(x,z))^2} dz}.
\end{equation}
This friction terms in equation \eqref{eq:Friction} have been adapted for the two-layer case from the formula used in \cite{khan2014modeling} for one-layer non-rectangular channels (Ch. V, p. 83). See \cite{kim2008two,castro2011numerical} for other approaches. We have adapted to our model the expression of the source term due to entrainment for flows along channels with constant width found in \cite{adduce2012gravity}. It is given by
\[
S_e = \frac{A_1}{\sigma_1} V_e,
\]
where $V_e$ is the entrainment velocity. Furthermore, we set $S_e$ proportional to $A_1$ to suppress entrainment in the absence of the heavier fluid in the internal layer. The cross sectional area in the external layer ($A_2$) is not small in the numerical tests in Section \ref{sec:CurrentsEntrainment} where entrainment is included.

Equations \eqref{eq:sw-1} and \eqref{eq:sw-3} express conservation of mass. The momentum equation \eqref{eq:sw-4} treats the elevation of the internal layer $w_1$ as a moving topography, resulting in a balance law with non-conservative products. The momentum equation \eqref{eq:sw-2} of the internal layer has a momentum exchange of the external layer on it, represented by non-conservative products. We note that in the limit $h_2 \rightarrow 0$ we recover the one layer shallow-water system. On the other hand, when $\sigma = \sigma(x)$ is independent of height (straight vertical walls), the pressure terms in \eqref{eq:p1} and \eqref{eq:p2} become
\begin{displaymath}
p_1 = g \frac{\sigma h_1^2}{2} + r g \sigma h_1 h_2 \hspace{0.5in} \text{and} \hspace{0.5in} p_2 = g \frac{\sigma h_2^2}{2}, 
\end{displaymath}
and the source terms in \eqref{eq:IntegralTerms} become
\begin{displaymath}
I_1 = \frac{h_1^2}{2} \frac{d\sigma}{dx} \hspace{0.7in} \text{and} \hspace{0.7in} I_2 = \frac{h_2^2}{2} \frac{d\sigma}{dx},
\end{displaymath}
resulting in the model presented in \cite{balbas2013non} in the absence of friction and entrainment.

Source terms in hyperbolic balance laws that do not depend on the derivatives of the solution variables do not modify the Rankine-Hugoniot jump conditions of weak solutions. However, the source terms in system \eqref{eq:sw} include $rg h_2 \partial_x A_1$ and $g h_2 \sigma_1 \partial_x w_1$ with non-conservative products that depend on derivatives of the wet areas. The jump conditions are modified by such terms leading to theoretical and numerical challenges \cite{AbgrallSmadar2010}. A theory of non-conservative products can give us a notion of weak solutions in those cases \cite{DalMasoetal1995}. Following the ideas in \cite{kurganov2009central} implemented for 2D and 1D two-layer shallow-water flows (with no width variation), we rewrite system \eqref{eq:sw} in a more convenient way to treat the non-conservative products. To that end, we decompose the hydrostatic pressures as
\[
p_1 = g \widehat w_2 A_1- g\int_B^{w_1} z \, \sigma(x,z) \, dz, \ \ \ \ \text{and} \ \ \ \ p_2 = g w_2 A_2 - g\int_{w_1}^{w_2} z \, \sigma(x,z) \, dz, 
\]
where $\widehat w_2 = B+h_1+r h_2$. Rewriting system \eqref{eq:sw} in terms of these, we obtain the alternative model
\begin{subequations} \label{eq:swW2Hat}
\begin{align}
\frac{\partial A_1}{\partial t} & + \frac{\partial Q_1}{\partial x} = S_e,  \\[0.1in]
\frac{\partial Q_1}{\partial t} & + \frac{\partial}{\partial x} \left(\frac{Q_1^2}{A_1} + g \widehat w_2 A_1 \right) = g \widehat w_2 \frac{\partial A_1}{\partial x}+ S_{f,1} +S_e u_1,  \\[0.1in]
\frac{\partial A_2}{\partial t}  & + \frac{\partial Q_2}{\partial x} = -r S_e,  \\[0.1in]
\frac{\partial Q_2}{\partial t}  & + \frac{\partial}{\partial x} \left(\frac{Q_2^2}{A_2}+ g w_2 A_2 \right) =  g w_2 \frac{\partial A_2}{\partial x}+ S_{f,2}-r S_e u_2.
\end{align}
\end{subequations}

We note that system \eqref{eq:swW2Hat} still has non-conservative products and it is equivalent to system \eqref{eq:sw} only when the associated solution is smooth. Therefore, both systems \eqref{eq:sw} and \eqref{eq:swW2Hat} present similar theoretical challenges. However, the non-conservative products in system \eqref{eq:swW2Hat} have $w_2$ and $\widehat w_2$ as factors, which are approximately constant in many situations where the rigid-lid approximation is valid, such as the case of internal waves. In the ideal case where such factors are constant, however, those non-conservative products become conservative, and system \eqref{eq:swW2Hat} therefore minimizes the numerical challenges when simulating two-layer flows. If the spatial derivatives of $A_1$ and $A_2$ in the non-conservative products are treated as flux gradients, then any consistent second order reconstruction of $w_2$ and $\widehat w_2$ leads to robust numerical schemes. Details are given in Section \ref{sec:scheme}. Other efforts to add stability to the two-layer shallow-water equations includes the three-layer approximation, as described in \cite{chertock2013three}.

\subsection{The Quasilinear Form} \label{eq:properties}

For one-layer shallow-water systems, one can find explicit expressions involving a speed of sound quantity. For two-layer flows, however, the eigenvalues have no explicit formulas but can be approximated by explicit expressions under suitable conditions. Nonetheless, the definition of speed of sound for each layer will still be helpful. For the external layer, this quantity is analogous to the one layer case, and it is given by
\begin{equation}
\label{eq:c_2}
c_2 = \sqrt{\frac{g A_2}{\sigma_2}},
\end{equation}
which, in practice, makes the internal layer a moving topography. The speed of sound for the internal layer depends on the width of both layers and it is given by
\begin{equation}
\label{eq:c_1}
c_1 = \sqrt{g \left( \epsilon + r \frac{\sigma_1}{\sigma_2} \right) \frac{A_1}{\sigma_1}} = \sqrt{g\left( r \frac{ A_1}{\sigma_2} + \epsilon  \frac{ A_1}{\sigma_1}\right)},
\end{equation}
where $\epsilon=1-r$.

Similarly to channels with straight walls, the system can be written in quasi-linear form. The details are in the following proposition.

\clearpage

\begin{proposition}
Either system, \eqref{eq:sw} or \eqref{eq:swW2Hat} can be written in quasilinear form, ${\bf W}_t + M({\bf W}; \sigma, B) {\bf W}_x = S({\bf W}; \sigma, B)$ with
\begin{subequations} \label{eq:semilinear}
\begin{align}
 \ & {\bf W}(x,t) = \begin{pmatrix} A_1 \\ Q_1\\ A_2 \\ Q_2 \end{pmatrix}, & \hspace{-2.4in} M({\bf W}; \sigma, B) = \begin{pmatrix} 0 & 1 & 0 & 0 \\ c_1^2 - u_1^2 & 2 u_1 & r g\dfrac{A_1}{\sigma_2} & 0 \\ 0 & 0 & 0 & 1 \\ c_2^2 & 0 & c_2^2-u_2^2 & 2 u_2 \end{pmatrix}, \\
\hspace{-2in} \text{and} \hspace{0.87in} & \ & \ \ \ \ \ \no \\
\ & S({\bf W}; \sigma, B) = \begin{pmatrix} S_e \\ c_1^2 \left( I_3-\sigma_B \frac{dB}{dx} \right) + S_{f,1} + S_e u_1 \\ -rS_e \\ c_2^2 \left( I_4-\sigma_B \frac{dB}{dx} \right) + S_{f,2} - rS_eu_2\end{pmatrix}, & \ 
\end{align}
\end{subequations}
where
\beq
I_3 = \int_B^{w_1}\sigma_x(x,z) \, dz, \ \ \ \ \ \ \ \text{and} \ \ \ \ \ \ I_4= \int_B^{w_2}\sigma_x(x,z) \, dz.
\eeq

\end{proposition}

\begin{proof}
The derivation of the quasilinear form follows from the fundamental theorem of calculus 
\beq
\label{eq:FTCcont}
\frac{\partial}{\partial x} \left( \int_{a(x)}^{b(x)} \sigma(x,z)dz\right) = \int_{a(x)}^{b(x)}\sigma_x(x,z)dz - \sigma(x,a(x))\frac{d}{dx}a(x) + \sigma(x,b(x))\frac{dB}{dx}(x), 
\eeq
and the relations
\begin{equation}
\frac{\partial w_1}{\partial x} = \frac{\frac{\partial A_1}{\partial x}-I_3+\sigma(x,B)\frac{dB}{dx}}{\sigma(x,w_1)}, \ \ \ \ \ \text{and} \ \ \ \ \ \frac{\partial w_2}{\partial x} = \frac{\frac{\partial}{\partial x} (A_1 + A_2)-I_4}{\sigma(x,w_2)}.
\end{equation}

\end{proof}

\subsection{Asymptotic Approximations of the Internal Eigenvalues as $\epsilon \to 0$}

The characteristic polynomial of the matrix $M$ in \eqref{eq:semilinear} is
\begin{equation}
\label{eq:CharPoly}
p(\lambda) = \left[(\lambda-u_1)^2-c_1^2 \right] \left[ (\lambda-u_2)^2-c_2^2\right]-r \frac{g A_1}{\sigma_2} c_2^2.
\end{equation}
Due to the last term, the eigenvalues have no explicit simple formulas. The following proposition provides us with a first order approximation for the internal eigenvalues.

\begin{proposition}
The two eigenvalues associated with the internal layer satisfy
\begin{equation}
\label{eq:InternalEig}
\lambda_{int}^{\pm} \approx \widehat u \pm \sqrt{\frac{\epsilon \frac{g A_1}{\sigma_1} c_2^2}{c_1^2+c_2^2} \left( 1- \frac{c_1^2}{\frac{g A_1}{\sigma_1}} \frac{(u_2-u_1)^2}{\epsilon (c_1^2+c_2^2)}\right)} \ \ \ \text{as} \ \ \ \epsilon \to 0,
\end{equation}
where 
\[
\widehat u = \frac{c_2^2 u_1+c_1^2 u_2}{c_2^2+c_1^2}
\]
is the convective velocity.
\end{proposition}

\begin{proof}

Expanding the characteristic polynomial we get
\[
c_2^2 (\lambda-u_1)^2 + c_1^2 (\lambda-u_2)^2-(\lambda-u_1)^2 (\lambda-u_2)^2 = c_1^2 c_2^2 - r \frac{g A_1}{\sigma_2} c_2^2 = \epsilon \frac{g A_1}{\sigma_1} c_2^2.
\]
Dividing by $\epsilon \frac{g A_1}{\sigma_1} c_2^2$, we get
\[
\frac{(\lambda-u_1)^2}{\epsilon \frac{g A_1}{\sigma_1} }+ \frac{c_1^2}{\frac{gA_1}{\sigma_1}} \frac{(\lambda-u_2)^2}{\epsilon c_2^2}-\epsilon \frac{(\lambda-u_1)^2}{\epsilon \frac{g A_1}{\sigma_1}} \frac{(\lambda-u_2)^2}{\epsilon c_2^2} = 1.
\]
The conditions for vanishing eigenvalues give us the composite Froude number
\begin{equation}
\label{eq:CompositeFroude}
G^2 = \frac{u_1^2}{\epsilon \frac{g A_1}{\sigma_1}}+\frac{c_1^2}{\frac{gA_1}{\sigma_1}}\frac{u_2^2}{\epsilon c_2^2}-\epsilon \frac{u_1^2}{\epsilon \frac{g A_1}{\sigma_1}} \frac{u_2^2}{\epsilon c_2^2},
\end{equation}
where $c_1$ and $c_2$ are given, respectively, by equations \eqref{eq:c_1} and \eqref{eq:c_2}. We note that $\frac{c_1^2}{gA_1/\sigma_1} \to \frac{\sigma_1}{\sigma_2}$ as  $\epsilon \to 0$. So, the first two terms in the above equation are order 1 and
\[
 \frac{(\lambda-u_1)^2}{\epsilon \frac{g A_1}{\sigma_1}}+\frac{c_1^2}{g A_1/\sigma_1}\frac{(\lambda-u_2)^2}{\epsilon c_2^2} \approx 1.
\]
The roots of the above quadratic polynomial are those given in \eqref{eq:InternalEig}.

\end{proof}

\subsection{Asymptotic Approximations of the External Eigenvalues as $u_2-u_1 \to 0$, $\epsilon \to 0$}

Approximate explicit formulas of the external eigenvalues can obtained when $u_2-u_1$ and $\epsilon$ are both small. In that case, the fluid behaves like a one layer shallow-water flow. The details are given in the following proposition.

\begin{proposition}
Assume $u_1,A_1,A_2,\sigma_1$ and $\sigma_2$ are all fixed. Let $\delta = (u_2-u_1)/\bar c$, and $\epsilon = 1-r$ as above. Let us assume that both parameters go to zero at the same rate. The external eigenvalues satisfy
\begin{equation}
\label{eq:ExternalEig}
 \lambda_{ext}^{\pm} = \overline u \pm \overline c 
 \mp 
\frac{\epsilon}{2} \frac{\frac{gA_1}{\sigma_2}}{\sqrt{\frac{g(A_1+A_2)}{\sigma_2}}} \left( 1- \frac{ \frac{g A_1}{\sigma_1} }{\frac{g (A_1+A_2)}{\sigma_2} }\right) + \bar c \;  O(\delta^2), \ \ \ \text{ as } \ \ \delta \to 0,
\end{equation}
where
\[
\overline u = \frac{A_1 u_1+A_2 u_2}{A_1+A_2}, \ \ \ \ \ \ \  \overline c = \sqrt{\frac{g (A_1+A_2)}{\sigma_2}}
\]
are the vertically averaged velocity and the speed of sound of a one-layer shallow-water system with total cross-sectional area $A_1+A_2$ and surface width $\sigma_2$.
\end{proposition}


\begin{proof}
 In what follows, we assume the parameters $A_1,A_2,u_1,\sigma_1,\sigma_2$ are all fixed. The only parameter that varies is $u_2 = u_1+\delta_1$, where $\delta_1 = \delta \;  \bar c$. We look for an expansion of the form
 \[
 \lambda = \lambda_o + \lambda_1 + \bar c \; O(\delta^2),
 \]
 with $\lambda_o$ independent of $\epsilon$ and $\delta_1$, and $\lambda_1 = O(\delta_1)$. Then
 \begin{align}
 0 = & \left[ (\lambda_o -u_1 + \lambda_1  )^2-c_1^2 \right] \left[ (\lambda_o -u_1 +\lambda_1- \delta_1 )^2 - c_2^2 \right] - r \frac{g A_1}{\sigma_2} c_2^2 + O(\delta_1^2) \no \\[0.15in]
=   & \left[ (\lambda_o-u_1)^2-\frac{g A_1}{\sigma_2} \right] \left[ (\lambda_o-u_1)^2-\frac{g A_2}{\sigma_2} \right] -\frac{g A_1}{\sigma_2} \frac{g A_2}{\sigma_2} + 2(\lambda_o-u_1) (\lambda_1-\delta_1 ) \left( (\lambda_o-u_1)^2-\frac{g A_1}{\sigma_2} \right)  \\[0.15in]
\ & + \left( 2 (\lambda_o-u_1) \lambda_1+\epsilon \left( \frac{g A_1}{\sigma_2}-\frac{gA_1}{\sigma_1} \right) \right) \left( (\lambda_o-u_1)^2-\frac{g A_2}{\sigma_2}  \right) + \epsilon \frac{g A_1}{\sigma_2} \frac{g A_2}{\sigma_2} + O(\delta_1^2). \no
\end{align}
 Collecting order 1 terms we get
 \[
   \left[ (\lambda_o-u_1)^2-\frac{g A_1}{\sigma_2} \right] \left[ (\lambda_o-u_1)^2-\frac{g A_2}{\sigma_2} \right] -\frac{g A_1}{\sigma_2} \frac{g A_2}{\sigma_2}  = 0,
 \]
 with solution 
 \[
 \lambda_o - u_1 = \pm \sqrt{\frac{g(A_1+A_2)}{\sigma_2}}.
 \]
 Collecting order $\delta = O(\epsilon)$ terms, we get
 \begin{equation}
 \begin{array}{l}
  2(\lambda_o-u_1) (\lambda_1-\delta_1 ) \left( (\lambda_o-u_1)^2-\frac{g A_1}{\sigma_2} \right)\\
 +
\left( 2 (\lambda_o-u_1) \lambda_1+\epsilon \left( \frac{g A_1}{\sigma_2}-\frac{gA_1}{\sigma_1} \right) \right) \left( (\lambda_o-u_1)^2-\frac{g A_2}{\sigma_2}  \right) + \epsilon \frac{g A_1}{\sigma_2} \frac{g A_2}{\sigma_2} = 0,
\end{array}
 \end{equation}
 with solution
 \[
 \lambda_1 = \delta_1 \frac{A_2}{A_1+A_2} \mp \epsilon \frac{\frac{gA_1}{\sigma_2}}{2\sqrt{\frac{g(A_1+A_2)}{\sigma_2}} } \left( 1-\frac{\sigma_2}{\sigma_1} \frac{A_1}{A_1+A_2} \right).
 \]
 Substituting $\lambda_o$ and $\lambda_1$ we get
 \[
 \lambda = \frac{A_1 u_1+ A_2 u_2}{A_1+A_2} \pm \sqrt{\frac{g(A_1+A_2)}{\sigma_2}} 
 \mp 
\frac{\epsilon}{2} \frac{\frac{gA_1}{\sigma_2}}{\sqrt{\frac{g(A_1+A_2)}{\sigma_2}}} \left( 1- \frac{ \frac{g A_1}{\sigma_1} }{\frac{g (A_1+A_2)}{\sigma_2} }\right) +  \bar c \; O(\delta^2),
 \]
 as desired.
\end{proof}

The approximation for the internal \eqref{eq:InternalEig} and external \eqref{eq:ExternalEig} eigenvalues indicates that the system is conditionally hyperbolic with an approximated condition given by
\begin{equation}
\label{eq:KelvinHelmholtz}
(u_2-u_1)^2 \le \frac{g A_1/\sigma_1}{c_1^2} \epsilon (c_1^2+c_2^2). 
\end{equation}
As documented in \cite{AbgrallKarni2009}, loss of hyperbolicity is associated with shear layer Kelvin-Helmholtz instabilities of the internal layer, and in such case the proposed balance law \eqref{eq:swW2Hat} is no longer valid to describe the flow. 
\subsection{Eigenvalue Bounds}

Assuming we are in the hyperbolic regime where all the eigenvalues are real, and denote by $\gamma_1^\pm, \gamma_2^\pm$ the roots of 
\[
(\gamma-u_1)^2-c_1^2 = \sqrt{r} \frac{gA_1}{\sigma_2}, \ \ \ \ \ \ (\gamma-u_2)^2-c_2^2 = \sqrt{r} c_2^2,
\]
given by
\begin{equation}
\label{eq:EigBounds}
\gamma_1^\pm = u_1\pm \sqrt{\sqrt{r}(1+\sqrt{r})\frac{gA_1}{\sigma_2}+\epsilon \frac{gA_1}{\sigma_1}}, \ \ \ \ \text{and} \ \ \ \ \gamma_2^\pm = u_2 \pm \sqrt{1+\sqrt{r}}c_2.
\end{equation}

The following proposition gives us the eigenvalue bounds. Such bounds will be used to define the one-sided local speeds, which in turn are used in the numerical scheme of Section \ref{sec:scheme}.

\begin{proposition}
\label{prop:EigBounds}
Let us assume that the eigenvalues of the coefficient matrix $A$ in equation \eqref{eq:semilinear} are all real. Then they all fall in the interval $[\min(\gamma_{1,2}^\pm), \max(\gamma_{1,2}^\pm)]$.
\end{proposition}

The proof of the previous proposition can be adapted to the present case of channels with arbitrary geometry from the original proof in \cite{AbgrallKarni2009}. We note that the values of $\gamma_k^\pm, k=1,2$ are always real, regardless of the hyperbolicity condition.

\begin{figure}[h!]
\center{
\includegraphics[width=0.8 \textwidth ]{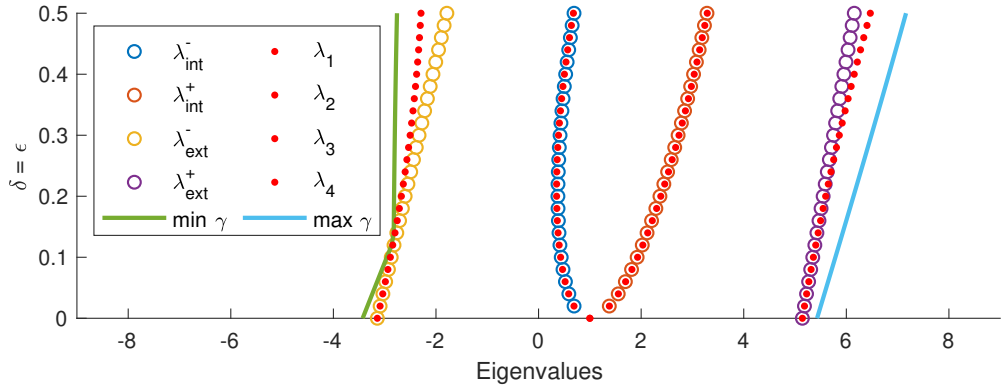}}
\caption{ \label{fig:Eigenvalues} The eigenvalues, the approximations \eqref{eq:ExternalEig} and \eqref{eq:InternalEig} and the bounds in Proposition \ref{prop:EigBounds} are shown for different values of $\delta = \epsilon$ from 0 to 0.5. Here $A_1 = 1.5, A_2 =2, u_1 = 1,  \sigma_1 = 1.4$, and $\sigma_2 = 2$}
\end{figure}

We have verified the approximations in equations \eqref{eq:ExternalEig} and \eqref{eq:InternalEig} numerically for specific values and have found that the approximations are excellent even if $\epsilon$ and $\delta$ are not too small. Figure \ref{fig:Eigenvalues} shows the eigenvalues, the approximations \eqref{eq:ExternalEig} and \eqref{eq:InternalEig} and the bounds in Proposition \ref{prop:EigBounds} for different values of $\delta = \epsilon$ from 0 to 0.5. The approximations are excellent and the bounds are not too far from the minimum and maximum eigenvalues.

\subsection{Steady-state Solutions} \label{sec:steady_states}

System \eqref{eq:sw} admits non-trivial steady-state solutions. The following proposition characterizes conditions satisfied by general smooth steady states. We also describe the internal waves, consisting of a moving internal layer at equilibrium and an external layer at rest, as well as steady states at rest for both layers.

\begin{proposition}
\label{prop:SS}
In the absence of friction and entrainment ($n_i=n_b=0, V_e=0$), smooth steady-state solutions of system \eqref{eq:sw} are characterized by the invariant quantities $Q_1,Q_2,E_1=\frac{1}{2}u_1^2+g(B+h_1+r h_2),E_2=\frac{1}{2}u_2^2+g(B+h_1+h_2)$. From those with given constant discharge $Q_1 \neq 0$, (piecewise) constant energy $E_1$, and constant free surface $w_2$, one can identify flows with internal waves that satisfy
\begin{subequations}  \label{eq:InternalWave}
\begin{align}
g(1-r)h_1^3+\left( g(1-r)B+gr w_2 -E_1 \right) h_1^2 +\frac{1}{2} Q_1^2 & = 0, \\ h_1 + h_2 + B & = w_2,  \\  u_1 & = Q_1/h_1, \\ u_2 & = 0,
\end{align}
\end{subequations}
and steady states of rest that satisfy
\begin{subequations}
\label{eq:SSRest}
\begin{align}
B + h_1 & = \text{Const.}, \\
h_2 & = \text{Const.}, \\
 u_1 = u_2 & = 0
\end{align}
\end{subequations}
\end{proposition}

\begin{proof}
We use the relation in \eqref{eq:FTCcont} to obtain
\[
\partial_x p_2 = \partial_x \left[ g\int_{w_1}^{w_2}(w_2-z)\sigma(x,z)dz\right]
= g\int_{w_1}^{w_2}(w_2-z)\sigma_x(x,z)dz + g A_1 \partial_x w_2 -g h_2 \sigma_1 \partial_x w_1.
\]
Taking the difference between the flux gradients and the source term in the external layer we have
\[
\partial_x \left[ A_2 u_2^2 + p_2 \right] -g I_{2} +gh_2\sigma_1 \partial_x w_1 = u_2 \partial_x Q_2 + A_2\partial_x \left[ \frac{1}{2}u_2^2+g w_2\right].
\]
For the internal layer, we notice that
\[
\partial_x p_1=g\int_{B}^{w_1}(w_1+rh_2-z)\sigma_x(x,z)dz+gA_1\partial_x(w_1+rh_2)+rgh_2 \sigma_1 \partial_x w_1-g(h_1+rh_2)\sigma_B \frac{dB}{dx},
\]
and using
\[
\partial_x A_2= \int_B^{w_1} \sigma_x(x,z) \, dz + \sigma_1 \partial_x w_1 - \sigma_B \frac{dB}{dx},
\]
we obtain
\[
\partial_x p_1-gI_{1}+gh_1 \sigma_B \frac{dB}{dx}-rgh_2\partial_x A_1 = u_1 \partial_x Q_1+A_1 \partial_x \left[ \frac{1}{2}u_1^2+g(w_1+rh_2) \right],
\]
which concludes the proof for the general smooth steady states.

If we set $u_2 = 0$, we obtain $w_2 = B+h_1+h_2 = Const.$ On the other hand, $B+h_1+r h_2 = (1-r) B+ (1-r) h_1 + r w_2$, which implies the conditions \eqref{eq:InternalWave}. 
\end{proof}

\subsection{Entropy Functions}

The existence of entropy functions and entropy inequalities have shown to be helpful in choosing the correct weak solution \cite{oleinik1957discontinuous,Audusse2004}. A numerical scheme that satisfies a fully discrete entropy inequality can be found in \cite{bouchut2008entropy}. To conclude this section and for the sake of completeness in the analysis of the hyperbolic model, in the following proposition we prove that system \eqref{eq:sw} admits one entropy function and entropy inequality that accounts for the contribution of both layers. 

\begin{proposition}
In the absence of entrainment ($V_e=0$), system \eqref{eq:sw} is endowed with an entropy function $\mathcal E = \mathcal  E_1+ r \mathcal E_2$, where $\mathcal E_1= A_1 E_1-p_1E$, and $\mathcal E_2= A_2 E_2-p_2$, and its physically relevant solutions satisfy the entropy inequality
\begin{equation}
\label{eq:Entropy}
\partial_t \mathcal E +\partial_x \left( u_1 (\mathcal E_1+p_1)+r \; u_2 (\mathcal E_2+p_2) \right) \le 0. 
\end{equation}
\end{proposition}

\begin{proof}
Adding friction or viscous terms is usually the first step in deriving entropy inequalities. However, system \eqref{eq:sw} already admits friction terms. Without loss of generality, we assume that the Manning coefficients are positive. Using the fundamental theorem of calculus we get the relation $\partial_t p_2 = g A_2 \partial_t w_2-g h_2 \sigma_1 \partial_t w_1$. Multiplying the momentum equation of the external layer by $u_2$, the resulting equation can be rewritten as
\[
\partial_t (A_2 E_2) + \partial_x (Q_2 E_2) = g \partial_t (A_2 w_2) + g w_2 \partial_x Q_2 - g\frac{n_i^2 \left| \frac{Q_1 A_1+Q_2 A_2}{A_1+A_2}\right|}{R^{4/3}} (u_2-u_2) u_2.
\]
Subtracting the two equations and using $\partial_t A_1= \sigma_1 \partial_t w_1$, we get
\begin{equation}
\label{eq:Entropy1}
\partial_t \mathcal E_2 + \partial_x (u_2 (\mathcal E_2+p_2)) =- g h_2 \partial_x Q_1- g\frac{n_i^2 \left| \frac{Q_1 A_1+Q_2 A_2}{A_1+A_2}\right|}{R^{4/3}} (u_2-u_1) u_2.
\end{equation}
On the other hand, the pressure in the internal layer satisfies $\partial_t p_1 = g A_1 \partial_t (w_1+r h_2) + r g h_2 \sigma_1 \partial_t w_1$. Multiplying the momentum equation of the internal layer by $u_1$, the resulting equation can be rewritten as
\begin{align}
\partial_t (A_1 E_1) + \partial_x (Q_1 E_1) = & \ g \partial_t (A_1 (w_1+r w_2)) + g (w_1+ r w_2) \partial_x Q_1 \no \\[-0.15in] \ & \ \\[-0.15in] & \ \ \ + r g\frac{n_i^2\left| \frac{Q_1 A_1+Q_2 A_2}{A_1+A_2}\right|}{R^{4/3}} (u_2-u_1) u_1 - g\frac{n_b^2 \left| \frac{Q_1 A_1+Q_2 A_2}{A_1+A_2} \right|}{R^{4/3}} u_1^2 . \no
\end{align}
Subtracting the two equations we get
\begin{equation}
\label{eq:Entropy2}
\partial_t \mathcal E_1 + \partial_x (u_1 (\mathcal E_1+p_1)) = r g h_2 \partial_x Q_1+ r g\frac{n_i^2 \left| \frac{Q_1 A_1+Q_2 A_2}{A_1+A_2}\right|}{R^{4/3}} (u_2-u_1) u_1 - g\frac{n_b^2 \left| \frac{Q_1 A_1+Q_2 A_2}{A_1+A_2} \right|}{R^{4/3}} u_1^2.
\end{equation}
Finally, using equations \eqref{eq:Entropy1} and \eqref{eq:Entropy2} and defining $\mathcal E= \mathcal  E_1+ r \mathcal E_2$, we get
\[
\partial_t \mathcal E + \partial_x \left( u_1 (\mathcal E_1+p_1)+r\; u_2 (\mathcal E_2+p_2) \right) = - r g\frac{n_i^2 \left| \frac{Q_1 A_1+Q_2 A_2}{A_1+A_2}\right|}{R^{4/3}} (u_2-u_1)^2 - g\frac{n_b^2 \left| \frac{Q_1 A_1+Q_2 A_2}{A_1+A_2} \right|}{R^{4/3}} u_1^2,
\]
which concludes the proof.
\end{proof}


\section{Numerical Scheme} \label{sec:scheme}

In this section we describe a central-upwind scheme to approximate the solutions of system \eqref{eq:swW2Hat}. The scheme must satisfy the well-balance and positivity-preserving properties. See \cite{balbas2013non,KP2007,kurganov2009central,KurganovTadmor2000} for more details on central and central-upwind schemes. We extend the scheme here to two-layer flows along channels with arbitrary cross sections. The non-conservative products on the right hand side of \eqref{eq:swW2Hat} and the more complex geometry of the channels considered here makes the extension of these schemes more challenging than those presented in the above references. 

The system now takes the form
\begin{equation} \label{eq:BL}
\frac{\partial {\vect W}}{\partial t} + \frac{\partial}{\partial x} {\vect F}({\vect W}; \sigma, B) = {\vect S}({\vect W}; \sigma, B) ,
\end{equation}
with
\begin{equation}
\label{eq:FS}
{\vect W} = \begin{pmatrix} A_1 \\ Q_1 \\ A_2 \\ Q_2 \end{pmatrix}, \ \ \ \ \ {\vect F}({\vect W}; \sigma, B) = \begin{pmatrix} Q_1 \\[0.04in] \displaystyle{\frac{Q_1^2}{A_1} + g \widehat{w}_2 A_1 } \\ Q_2 \\[0.04in]  \displaystyle{\frac{Q_1^2}{A_1} + g w_2 A_2 } \end{pmatrix}, \ \ \text{ and } \ \ \ 
{\bf S}({\vect W}; \sigma, B) = \begin{pmatrix} S_e \\[0.05in] \displaystyle{g \widehat{w}_2 \partial_x A_1\,} + S_{f,1} + S_e u_1  \\[0.05in] -rS_e \\[0.05in] \displaystyle{g w_2 \partial_x A_2}+S_{f,2} - rS_e u_2\end{pmatrix}.
\end{equation}


\subsection{Discretization of the Channel's Geometry}

In order to arrive at a second order scheme for \eqref{eq:BL} - \eqref{eq:FS}, we shall first address the discretization of the channel geometry given by the width function $\sigma(x,z)$ and the bottom topography $B(x)$. Since our approximate solution of the system \eqref{eq:BL} - \eqref{eq:FS} will be realized as the cell averages $\overline{\vect W}_j$ of the flow quantities over the grid cells $I_j := [x_{j-\half},x_{j+\half}]$, where $\Delta x$ is the spatial scale in the direction of the flow, $x_{j\pm \half} = x_j\pm\frac{\Delta x}{2}$ and $x_j$ is the center of the grid cell, we begin by sampling the bottom topography and width of the channel at the points $x_{j + \half}$. From these samples, we define the cell average of the bottom topography as
\beq
\overline{B}_j := \frac{B_{j+\half} + B_{j - \half}}{2}.
\eeq
And similarly, we define the average width functions
\beq
\overline{\sigma}_j(z) := \frac{\sigma(x_{j+\half},z) + \sigma(x_{j - \half},z)}{2}.
\eeq
Using these, we build a piecewise trapezoidal approximation of the cross-sections of the channel at each cell center $x_j$
\beq
\widetilde{\sigma}_j(z) = \overline{\sigma}_j(z_l) + \frac{\overline{\sigma}_j(z_{l+1}) - \overline{\sigma}_j(z_l)}{\Delta z}(z - z_{l}) \ \ \ \ \text{for} \ \ \ \ z_l \leq z \leq z_{l+1},
\eeq
where $\Delta z$ is a fixed spatial scale in the vertical direction (flow height).

These piecewise (bi-)linear functions allow us to approximate the spatial derivatives of $B(x)$ and $\sigma(x,z)$ with centered divided differences and are consistent with the second order accuracy sought for the scheme (see left panel of Figure \ref{fig:reconstruction}). Unless otherwise noted, in the numerical experiments we use a resolution of $\Delta z = 0.01$.


\subsection{The Semi-discrete Central-upwind Scheme}

With the above partition of the solution grid and discretization of the channel geometry, we denote by $\bv_j(t)$ the computed cell average of $\vect W(x,t)$ over the cell $I_j$,
\beq \label{eq:cell_average}
\vect  \bv_j(t)=\frac{1}{\Delta x} \int_{x_{j-\frac{1}{2}}}^{x_{j+\frac{1}{2}}}\vect  W(x,t) \, dx.
\eeq
Integrating equation \eqref{eq:swW2Hat} over each cell $I_j$, we obtain the semi-discrete formulation
\beq \label{eq:discrete_balance}
\frac{d}{dt}\vect  \bv_j(t)+\frac{1}{\Delta x}\left( \vect F(\vect W(x_{j+\frac{1}{2}},t))-\vect F(\vect  W(x_{j-\frac{1}{2}},t)) \right)=\frac{1}{\Delta x} \int_{x_{j-\frac{1}{2}}}^{x_{j+\frac{1}{2}}}\vect S(\vect W(x,t),x) \ dx,
\eeq
which is approximated by
\begin{equation}
\label{eq:ODESW}
\frac{d}{dt} \vect \bv_{j}(t)=-\frac{\vect H_{j+\frac{1}{2}}-\vect H_{j-\frac{1}{2}}}{\Delta x}+\frac{1}{\Delta x} \int_{x_{j-\frac{1}{2}}}^{x_{j+\frac{1}{2}}} \vect S(\vect W,x) \, dx.
\end{equation}

The flux at the cell interfaces, $\vect F(\vect W(x_{j\pm \frac{1}{2}}),t)$, is approximated by the numerical flux $\vect H_{j\pm \frac{1}{2}}(t)$ in \cite{kurganov2001semidiscrete} 
\begin{equation}
\label{eq:NumFluxH}
\vect H_{j\pm \frac{1}{2}}(t) = \frac{a_{j\pm \frac{1}{2}}^{+} \vect F\left(\vect  W_{j\pm \frac{1}{2}}^-(t) \right)-a_{j\pm \frac{1}{2}}^{-} \vect F\left( \vect W_{j\pm \frac{1}{2}}^+(t)\right)  }{a_{j\pm \frac{1}{2}}^+-a_{j\pm \frac{1}{2}}^-}+\frac{a_{j\pm \frac{1}{2}}^+a_{j \pm \frac{1}{2}}^-}{a_{j\pm \frac{1}{2}}^+-a_{j\pm \frac{1}{2}}^{-}} \left( \vect W_{j\pm \frac{1}{2}}^+(t)-\vect W_{j\pm \frac{1}{2}}^-(t) \right).
\end{equation}
Here, $\vect W_{j\pm \frac{1}{2}}^\pm(t)$ is recovered from the cell averages via a non-oscillatory piecewise polynomial reconstruction
\beq \label{eq:VInerface}
\vect W_{j+\frac{1}{2}}^- := \vect p_{j}(x_{j+\frac{1}{2}}), \ \ \ \text{ and } \ \ \ \vect W_{j+\frac{1}{2}}^+:=\vect p_{j+1}(x_{j+\frac{1}{2}}).
\eeq
Furthermore, the one-sided local speeds in this scheme are approximated using the eigenvalues' bounds given by \eqref{eq:EigBounds}:
\beq
\label{eq:SidedVel}
a_{j\pm \frac{1}{2}}^+ = \max\{ \gamma_{1,j\pm \frac{1}{2}}^+, \gamma_{2,j\pm \frac{1}{2}}^+,0\}, \ \ \ \  \text{ and } \ \ \ \
a_{j\pm \frac{1}{2}}^- = \min\{ \gamma_{1,j\pm \frac{1}{2}}^-, \gamma_{2,j\pm \frac{1}{2}}^-,0\}.
\eeq
We note that in the case where the eigenvalues are real, the inequalities $0 \le a_{j\pm \half}^+ \ge \lambda_{k,j\pm \half}^\pm$, $0 > a_{j\pm \half}^- \le \lambda_{k,j\pm \half}^\pm$ hold for $k = ext, int$. On the other hand, the formulas in equation \eqref{eq:SidedVel} are always real valued even if hyperbolicity is lost. Furthermore, $a_{j\pm \half}^+ - a_{j\pm \half}^- >0 $ is always positive unless $h_{1,j\pm\half}^\pm,u_{1,j\pm \half}^\pm,h_{2,j\pm\half}^\pm,u_{2,j\pm \half}^\pm$ all vanish in a dry state with ``no fluid motion''. Near dry states, we use the regularization of the velocity given by  \eqref{eq:regularized_u} below, and it is usually not negligible. The condition $a_{j\pm \frac{1}{2}}^+ \ge \max(u_{j\pm \frac{1}{2}}^+,u_{j\pm \frac{1}{2}}^-), a_{j\pm \frac{1}{2}}^- \le \min(u_{j\pm \frac{1}{2}}^+, u_{j\pm \frac{1}{2}}^-)$ is automatically satisfied given the expressions in equation \eqref{eq:EigBounds}, which is important for the positivity-preserving property.

\subsection{Positivity Preserving Non-oscillatory Reconstruction} \label{sec:reconstruction}

In order to recover the interface point values $\vect W_{j \pm \half}^\mp(t)$ from the cell averages $\bv_j(t)$, we seek a piecewise polynomial reconstruction at cell $j$ given by
\beq \label{eq:minmodPoly}
\vect p_{j}(x) = \bv_j + \vect W_j'(x-x_j),
\eeq \label{eq:minmodapoly}
with the {\it limited} slopes $\vect W_j'$ calculated as, \cite{vanLeer1997},
\beq \label{eq:minmodSlope}
\vect W_j' = \frac{1}{\Delta x}\text{minmod} (\alpha \Delta_{-}  \overline{\vect{W}}_j, \ \Delta_0  \overline{\vect {W}}_j, \ \alpha \Delta_+  \overline{ \vect{W}}_j), \text{ where } 1\le \alpha < 2, \text{ and}
\eeq
\bdm
\text{minmod}(x_1,x_2,x_3,\ldots,x_k)=
\left\{
\begin{array}{lcl}
\min_j (x_j) & \text{ if } & x_j > 0 ~\forall j \\
\max_j(x_j) & \text{ if } & x_j < 0 ~\forall j \\
0 &  & \text{otherwise}
\end{array}.
\right.
\edm

This non-oscillatory reconstruction is applied directly to the discharges $Q_1$ and $Q_2$. However, in order to ensure the positivity of the water layers and recognize steady states of rest \eqref{eq:SSRest}, the point values of the areas $A_1$ and $A_2$ are obtained by first reconstructing the total elevations of the internal and external layers, $w_1 = B + h_1$ and $w_2 = w_1 + h_2$ respectively. These, unlike $h_1$ and $h_2$, will remain constant for the steady state of rest, so their reconstruction will also render constant point values within and across grid cells, and help us achieve well balance as well as prevent the onset of spurious oscillations.

In order to reconstruct the interface point values of $A_{1, j \mp \half}^{\pm}$ and $A_{2, j \mp \half}^{\pm}$ from their cell averages $\overline{A}_{1,j}$ and $\overline{A}_{2,j}$, we proceed as follows:
\bit
\item Determine the values $z_{l_1}$ and $z_{l_2}$ along the partition $\{z_l\}$ of the vertical axis such that
\beq \label{eq:deconvolve1}
\int_{B_j}^{z_{l_1}} \widetilde{\sigma}_j(z) \, dz \leq \overline{A}_{1,j} < \int_{B_j}^{z_{l_1 + 1}} \widetilde{\sigma}_j(z) \, dz
\eeq
and
\beq 
\int_{B_j}^{z_{l_2}} \widetilde{\sigma}_j(z) \, dz \leq \overline{A}_{1,j} + \overline{A}_{2,j} < \int_{B_j}^{z_{l_2 + 1}} \widetilde{\sigma}_j(z) \, dz.
\eeq \label{eq:deconvolve2}
\item Find the height $\delta w_{1,j}$ of the next trapezoid in the discretization of the $j^{th}$ cross-section of the channel so that
\beq \label{eq:dw1}
\int_{z_{l_1}}^{z_{l_1} + \delta w_{1,j}} \widetilde{\sigma}_j(z) \, dz = \overline{A}_{1,j} - \int_{B_j}^{z_{l_1}} \widetilde{\sigma}_j(z) \, dz,
\eeq
and set $\overline{w}_{1,j} := z_{l_1} + \delta w_{1,j}$.
\item Find the height $\delta w_{2,j}$ so that
\beq \label{eq:dw2}
\int_{z_{l_2}}^{z_{l_2} + \delta w_{2,j}} \widetilde{\sigma}_j(z) \, dz = \overline{A}_{1,j} + \overline{A}_{2,j} - \int_{B_j}^{z_{l_2}} \widetilde{\sigma}_j(z) \, dz,
\eeq
and set $\overline{h}_{2,j} := z_{l_2} + \delta w_{2,j} - \overline{w}_{1,j}$ and $\overline{w}_{2,j} : =\overline{h}_{2,j} + \overline{w}_{1,j}$.
\eit

Note that the left hand side of \eqref{eq:dw1} and \eqref{eq:dw2} are, respectively, increasing functions of $\delta w_1$ and $\delta w_2$, and because of the linearity of $\sigma_j(z)$, the equations are quadratic. Therefore, $\delta w_{1,j}$ and $\delta w_{2,j}$ correspond, each, to one of the solutions of a quadratic equation. As an alternative, we may also determine their values using bisection so as to pick the right solution, which works well here because the cross sectional area is an increasing function of height. 

Once the heights corresponding to the cell averages of the areas are recovered, following the ideas of \cite{kurganov2007second}, we reconstruct their interface values via the minmod reconstruction \eqref{eq:minmodPoly} -- \eqref{eq:minmodSlope}, but ensuring their positivity with the following correction of the reconstructed values where needed (see right panel of Figure \ref{fig:reconstruction}):

\bit

\item For the internal layer, if $w_{1,j - \half}^+ < B_{j - \half}$, then set 
\beq
w_{1, j - \half}^+ := B_{j - \half}+\delta_B, \ \ \  w'_{1,j} := 2(\overline{w}_{1,j} - B_{j - \half} - \delta_B), \ \ \ \text{and} \ \ \ w_{1,j+\frac{1}{2}}^- := \overline{w}_{1,j} + \half w'_{1,j},
\eeq
for some small $\delta_B$, or if $w_{1,j+\frac{1}{2}}^- < B_{j+\frac{1}{2}}$, then set 
\beq
w_{1,j+\frac{1}{2}}^- := B_{j + \frac{1}{2}} + \delta_B, \ \ \ w'_{1,j}  :=  2(B_{j + \half} + \delta_B - \overline{w}_{1,j}), \ \ \ \text{and} \ \ \  w_{1,j - \half}^+ := \overline{w}_{1,j} - \half w'_{1,j}.
\eeq
\item And for the external layer, if $w_{2, j - \half}^+ < w_{1, j - \half}^+$, then set
\beq
w_{2, j - \half}^+ := w_{1, j - \half}^+ + \delta_B, \ \ \  w'_{2,j} := 2(\overline{w}_{2,j} - w_{1, j - \half}^+ - \delta_B), \ \ \ \text{and} \ \ \ w_{2, j + \half}^- := \overline{w}_{2,j} + \half w'_{2,j},
\eeq
or if $w_{2, j + \half}^- < w_{1, j + \half}^- $, then set 
\eit
\beq
w_{2, j + \half}^- := w_{1, j + \half}^- + \delta_B, \ \ \ w'_{2,j}  :=  2(w_{1, j + \half}^- + \delta_B - \overline{w}_{2,j}), \ \ \ \text{and} \ \ \  w_{2, j - \half}^+ := \overline{w}_{2,j} - \half w'_{2,j}.
\eeq
This yields
\beq
h_{1,j \mp \half}^\pm := w_{1,j \mp \half}^\pm-B_{j \mp \half} > 0, \ \ \ \text{and}  \ \ \ \ \ \ h_{2,j \mp \half}^\pm := w_{2, j \mp \half}^\pm - w_{1, j \mp \half}^\pm > 0.
\eeq

\noindent See \cite{KP2007,balbas2014positivity} for more details. For the numerical experiments presented in Sections \ref{sec:LockExchange} and \ref{sec:CurrentsEntrainment}, we use $\delta_B = 10^{-3}$.

\begin{figure}[h!]
\center{
\includegraphics[scale=0.6]{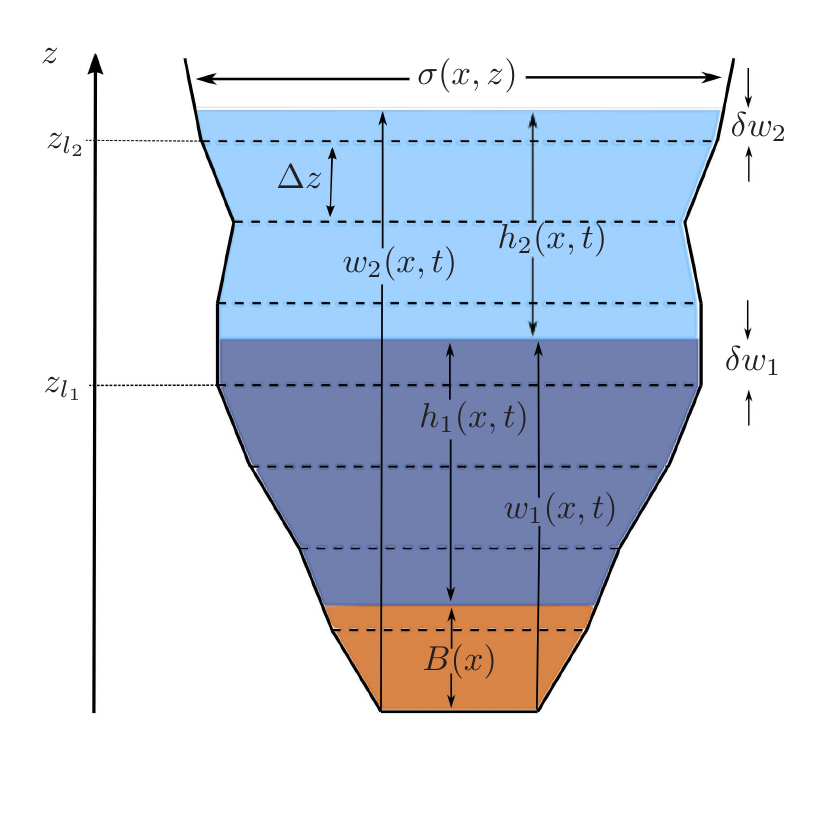} \hspace{0.3in}
\includegraphics[scale=0.6]{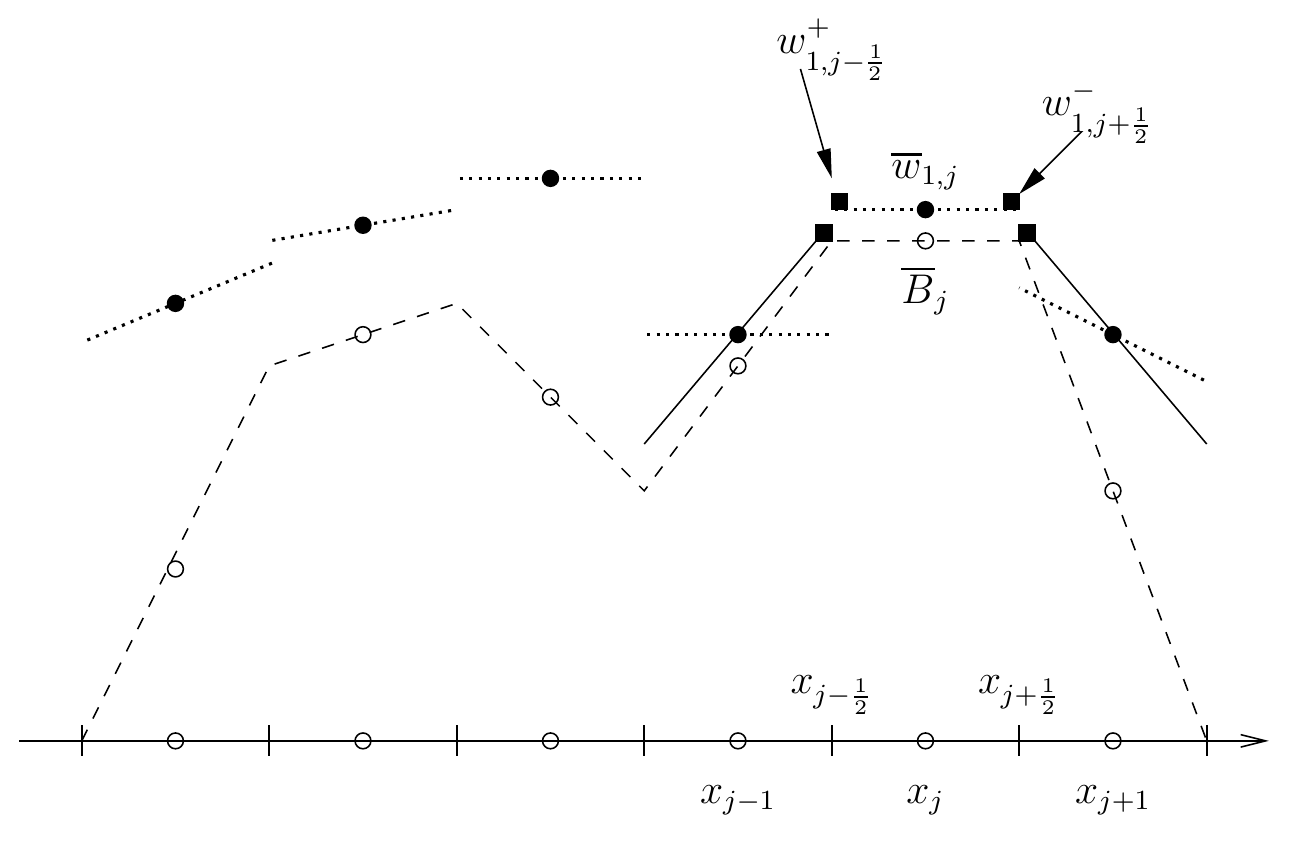}}
\caption{ \label{fig:reconstruction}Left: Piecewise trapezoidal discretization of a channel cross-section showing the quantities involved in the reconstruction of the interface point values of $A_1$ and $A_2$ described by \eqref{eq:deconvolve1} -- \eqref{eq:A_interface}. Right: Positivity preserving reconstruction of $w_1 = B + h_1$.} \label{fig:reconstruction}
\end{figure}

The pointvalues of the wet areas, $A^{\pm}_{1,j \mp \half}$ and $A^{\pm}_{2,j \mp \half}$, are then recovered from these by integrating the interface width,
\beq \label{eq:A_interface}
A_{1,j\mp\half}^{\pm}=\int_{B_{j\mp\half}}^{w_{1,j\mp \half}^\pm} \widetilde{\sigma}_{j \mp \half}(z) \, dz \ \ \ \ \ \text{and} \ \ \ \ \ A_{2,j\mp\half}^{\pm}=\int_{w_{1,j\mp \half}^\pm}^{w_{2,j\mp \half}^\pm} \widetilde{\sigma}_{j \mp \half}(z) \, dz,
\eeq
where
\beq
\widetilde{\sigma}_{j \mp \half }(z) = \sigma(x_{j \mp \half}, z_l) + \frac{\sigma(x_{j \mp \half}, z_{l+1}) - \sigma(x_{j \mp \half},z_l)}{\Delta z}(z - z_l) \ \ \ \ \text{for} \ \ \ \ z_l \leq z \leq z_{l+1}
\eeq
is a piecewise trapezoidal approximation of the channel cross-section at the cell interfaces $x_{j \mp \half}$. These guarantee the positivity of the water layers and will remain constant across cell interfaces for steady states of rest, where the total elevation of the water layers, $w_1 = h_1 + B$ and $w_2 = w_1 + h_2$, remain constant.
\vskip 10 pt
\noindent
{\bf Regularization of Flow Velocity and Discharge for Small $A_1,A_2$\\}

Although the positivity of depth in each layer is preserved, these point-values may still be very small and may lead to large values of the velocity of the flow, $u_1,u_2$. We can prevent this with the regularization technique suggested in \cite{KP2007},
\beq \label{eq:regularized_u}
u_{k,j\pm \frac{1}{2}}^\pm=\frac{\sqrt{2}\, Q_{k,j\pm \frac{1}{2}}^\pm A_{k,j\pm \frac{1}{2}}^{\pm}}{\sqrt{ \left( A_{k,j\pm \frac{1}{2}}^{\pm} \right)^4+\max \left( (A_{k,j\pm \frac{1}{2}}^\pm)^4,\delta_A \right)}}
\eeq
for $k=1,2$. The discharges $Q_{k,j\pm \frac{1}{2}}^\pm$ are then recalculated as
\beq
Q_{k,j\pm \frac{1}{2}}^\pm = A_{k,j\pm \frac{1}{2}}^{\pm} \cdot u_{k,j\pm \frac{1}{2}}^\pm
\eeq
to ensure conservation. The value of $\delta_A$ was empirically determined, usually choosing $\delta_A = 10^{-12}$ in this paper.

\subsection{Well Balance} \label{sec:WellBalance}

Flows described by the balance law \eqref{eq:swW2Hat} admit steady-state solutions of rest satisfying \eqref{eq:SSRest}. Numerical schemes for approximating the solutions of \eqref{eq:swW2Hat} that are capable of capturing these equilibrium solutions are said to satisfy the {\it well balance} property.

The authors of \cite{KurganovLevy2002} and \cite{KP2007}, where shallow-water models for flows along channels with constant width is discussed, propose to reconstruct the variables $w_1$ and $w_2$, which remain constant for steady states of rest, and not the heights $h_1$ and $h_2$, that, despite the water layers remaining flat, may vary across cells due to changes in the bottom topography. While our model \eqref{eq:swW2Hat} describes the evolution of wet areas, not heights, $Q_{1,j\mp \half}^\pm, Q_{2,j\mp \half}^\pm$, $w_{1,j\mp \half}^\pm, w_{2,j\mp \half}^\pm$ still characterize the steady states of rest, for which they are either zero or constant (as is $\widehat w_2 = w_1 + r h_2$). The positivity preserving minmod reconstruction of point values proposed in \S \ref{sec:reconstruction} above will recognize and preserve the constant values of these variables and will recover the point values $A_{1,j \mp \half}^\pm$ and $A_{2, j \mp \half}^\pm$ in \eqref{eq:A_interface} in a manner consistent with those steady states, allowing us to achieve well balance.

The rest of the variables satisfy $(\cdot)_{j\mp \half}^+ = (\cdot)_{j\mp \half}^-$ in the case of a steady state at rest. In order to obtain such discretization of the source terms, we begin by writing the numerical flux for $Q_1$ under these rest conditions,
\begin{dmath} 
\label{eq:fluxQ1}
- \frac{1}{\Delta x} \left[ H_{j + \half}^{Q_1} - H_{j - \half}^{Q_1}\right] = - \frac{1}{\Delta x} \left( g \widehat w_{2,j}  ( A_{1,j + \half}^- - A_{1, j - \half}^+ ) \right),
\end{dmath}
and seek a discretization of the cell average of the corresponding source term that matches this numerical flux $S^{Q_1} (x, {\vect W}, {\vect W}_x)  = g \widehat w_2 \partial_x A_1$. A consistent discretization is given in the following proposition. 
\begin{proposition}
\label{prop:SourceDisc}
Let us consider the following discretization of the source terms in \eqref{eq:FS}
\begin{align}\label{eq:SDisc}
\ & \hspace{-0.1in} \frac{g}{\Delta x} \int_{x_{j - \half}^+}^{x_{j + \half}^-} \widehat w_2 \partial_x A_1 \, dx \approx g \widehat w_{2,j} \frac{A_{1,j+\half} - A_{1,j-\half}}{\Delta x}, \; \; \;
\frac{g}{\Delta x} \int_{x_{j - \half}^+}^{x_{j + \half}^-}  w_2 \partial_x A_2 \, dx \approx g  \bar w_{2,j} \frac{A_{2,j+\half} - A_{2,j-\half}}{\Delta x}, \no \\[0.15in]
\ & \hspace{-0.1in} \frac{g}{\Delta x} \int_{x_{j - \half}^+}^{x_{j + \half}^-} S_e u_1 \, dx = \frac{\bar A_{1,j}}{\sigma_{1,j}} \; V_e \; \bar u_{1,j} , \; \; \;
\frac{g}{\Delta x} \int_{x_{j - \half}^+}^{x_{j + \half}^-} S_e \, dx = \frac{\bar A_{1,j}}{\sigma_{1,j}} \; V_e , \;  \;
\frac{g}{\Delta x} \int_{x_{j - \half}^+}^{x_{j + \half}^-} S_e u_2 \, dx = \frac{\bar A_{1,j}}{\sigma_{1,j}} \; V_e  \; \bar u_{2,j}, \no \\[-0.05in] 
\ & \ \\[-0.1in] 
\ & \hspace{-0.1in} \frac{g}{\Delta x} \int_{x_{j - \half}^+}^{x_{j + \half}^-} S_{f,1} \, dx = - rg\frac{n_i^2 \left| \frac{\bar Q_{1,j} \bar A_{1,j}+\bar Q_{2,j} \bar A_{2,j}}{\bar A_{1,j}+\bar A_{2,j}} \right| }{\bar R_j^{4/3}} (\bar u_{1,j}-\bar u_{2,j}) -  g\frac{n_b^2  \left| \frac{\bar Q_{1,j} \bar A_{1,j}+\bar Q_{2,j} \bar A_{2,j}}{\bar A_{1,j}+\bar A_{2,j}} \right|}{\bar R_j^{4/3}} \bar u_{1,j}, \no \\[0.15in]
\ & \hspace{-0.1in} \frac{g}{\Delta x} \int_{x_{j - \half}^+}^{x_{j + \half}^-} S_{f,2} \, dx = - g\frac{n_i^2  \left| \frac{\bar Q_{1,j} \bar A_{1,j}+\bar Q_{2,j} \bar A_{2,j}}{\bar A_{1,j}+\bar A_{2,j}} \right|}{\bar R_j^{4/3}} (\bar u_{2,j}-\bar u_{1,j}). \no
\end{align}
where
\[
A_{k,j \pm \half} = \frac{a_{j\pm \half}^+ A_{k,j \pm \half}^- - a_{j\pm \half}^- A_{k,j \pm \half}^+}{a_{j\pm \half}^+ - a_{j\pm \half}^-}, k =1,2, \; 
\text{ and }
\bar R_j = \frac{\bar A_{1,j}+\bar A_{2,j}}{\sigma_{B_j}+\int_{B_j}^{\bar w_{2,j}} \sqrt{4+(\partial_z \sigma_j(z))^2} dz}.
\]
Then the numerical scheme given by equation \eqref{eq:ODESW} satisfies the well-balance property. That is, it recognizes steady states of rest.
\end{proposition}

The averages in the computations $A_{k,j \pm \half}$ are done consistently with the discretization of the numerical fluxes in equation \eqref{eq:NumFluxH}, which is relevant for the cases where $w_2, \hat w_2$ are constant.  

\subsection{Evolution}
Once the interface values, the numerical fluxes and the average of the source terms have been calculated, the ODE system \eqref{eq:ODESW} is integrated in time using the second order Strong Stability Preserving Runge-Kutta scheme \cite{GottliebShuTadmor2001},
\begin{subequations} \label{eq:RKEvolution}
\begin{align}
\vect W^{(1)} =  & \overline{\vect W}(t)+\Delta t \, \vect C[\overline{\vect W}(t)] \label{eq:euler} \\[0.1in] 
\vect W^{(2)} = & \ \half \overline{\vect W}(t) + \half \left(\vect W^{(1)} + \Delta t \, \vect C[\vect W^{(1)}] \right) \label{eq:RK2ndstage} \\[0.1in] 
\overline{\vect W}(t+\Delta t) := & \ \vect W^{(2)},
\end{align}
\end{subequations}
with the Runge-Kutta fluxes
\begin{equation}
\label{eq:NumFlux}
\vect C[\vect W(t)]_j = -\frac{\vect H_{j+\frac{1}{2}}(v(t))-\vect H_{j-\frac{1}{2}}(v(t))}{\Delta x} + \overline{\vect S}_j(t),
\end{equation}
with $\overline{\vect S}_j(t)$ described by equation \eqref{eq:FS}, and discretized by equation \eqref{eq:SDisc}. The time step $\Delta t$ is determined so as to satisfy the {\it CFL} restriction
\beq\label{eq:cfl}
 \nu = \Delta t \max \left[ \frac{a}{\Delta x} \max_{k=1,2, j=1:N}\frac{A_{k,j+\half}^-+A_{k,j-\half}^+}{2\overline{A}_{k,j}} + \tau_e , 5 \tau_f \right] \le \half,
\eeq
where $k=1,2$ denote the corresponding quantities for the internal and external layers respectively. Here, $\tau_e$ is given by 
\begin{equation}
\label{eq:tau_e}
\tau_e = \min\left( 0, \min_{j=1:N} \frac{\bar V_{e,j}}{\sigma_{1,j}} , \min_{j=1:N} \frac{\bar A_{1,j}}{\bar A_{2,j}}\frac{\bar V_{e,j}}{\sigma_{1,j}} \right),
\end{equation}
and
\begin{equation}
\label{eq:tau_friction}
\tau_f = r\; g \; \max(n_i,n_b)^2 \max_{j=1:N} \frac{\left| \frac{\bar Q_{1,j} \bar A_{1,j}+\bar Q_{2,j} \bar A_{2,j}}{(\bar A_{1,j}+\bar A_{2,j})^2} \right| }{\bar R_j^{4/3}}.
\end{equation}
We note that in channels with straight walls, $\frac{A_{k,j+\half}^-+A_{k,j-\half}^+}{2\bar A_{k,j}} =1$ by construction. In the case of general channels, that quantity is only approximately 1, so the extra condition is not too restrictive. Furthermore, $\tau_e$ is non-positive and such extra restriction is due to entrainment. We also note that we are considering entrainment only when the external layer $A_2$ is away from zero, and the above definition is valid. The quantity in \eqref{eq:tau_friction} is a frequency scale given by the friction term. The condition in \eqref{eq:cfl} guarantees that the friction is resolved in the sense that we include at least 10 time steps during each time period of length $1/\tau_f$. 

\vskip 6pt

One can show that the CFL restriction in equation \eqref{eq:cfl} also guarantees the positivity-preserving property. That is, the positivity of $A_1$ and $A_2$ are preserved in the numerical evolution. The details are given in the following proposition.

\begin{proposition} 
\label{prop:positivity}
Consider the scheme \eqref{eq:ODESW} -- \eqref{eq:NumFluxH} with the reconstruction algorithm described in \S \ref{sec:reconstruction} and the discretization of the source term given by proposition \eqref{prop:SourceDisc}.  If the cell averages $\overline{A}_k(t),k=1,2$ are such that
\bdm
\overline{w}_{1,j}(t)\ge \frac{B_{j-\frac{1}{2}}+B_{j+\frac{1}{2}}}{2}, \hspace{0.4in} \overline{w}_{2,j} (t) \ge \overline{w}_{1,j} \ \ \forall j,
\edm
then the cell averages $\bar A(t+\Delta t)$ as evolved with time evolution given by equation \eqref{eq:RKEvolution}, under the {\it CFL} limitation \eqref{eq:cfl}, and with one-sided local speeds given by equation \eqref{eq:SidedVel}  yields
\bdm
\overline{A}_{k,j}(t+\Delta t) \ge 0 \ \ \ \forall j, \ k=1,2.
\edm
\end{proposition}

The proof can be done for each layer as in \cite{balbas2014positivity}, and by adding the contribution of entrainment. It will not be repeated here.

\section{Numerical results} \label{sec:Results}

\noindent
{\bf Numerical Setup and Boundary Conditions}\\

In this section we present numerical solutions for a range of problems illustrating the properties of our central-upwind scheme in a variety of flow situations. For all the results presented below the value of the acceleration of gravity is taken as $g = 9.81$, and the time step, $\Delta t$, satisfies the $CFL$ restriction \eqref{eq:cfl}. Unless otherwise mentioned, the computations below were performed using 200 grid cells and $\nu = 0.45$. In all numerical tests, $x,z,\sigma$ are in units of $\text{ m}$, $A_1,A_2$ in $\text{m}^2$, $Q_1,Q_2$ in $\text{ m}^3 \text{s}^{-1}$, $u_1,u_2,V_e$ in $\text{m}\text{s}^{-1}$, $g$ in $\text{m}^2 \text{ s}^{-1}$, and $n_i,n_b$ in $\text{ s m}^{-1/3}$.

At the left (right) boundary, outflow occurs when $\gamma_1^- < 0, \; \gamma_2^- < 0$ ($\gamma_1^+ > 0, \; \gamma_2^+ > 0$) respectively. Inflow is considered to occur when such conditions are not satisfied. Unless otherwise noted, we extrapolate the variables $w_1,w_2,Q_1,Q_2$ at outflow, and prescribe $w_1,w_2,Q_1,Q_2$  at inflow (with the values given by the initial conditions).

\subsection{Riemann Problem}

\begin{figure}[h!]
\center{
\includegraphics[scale=0.49]{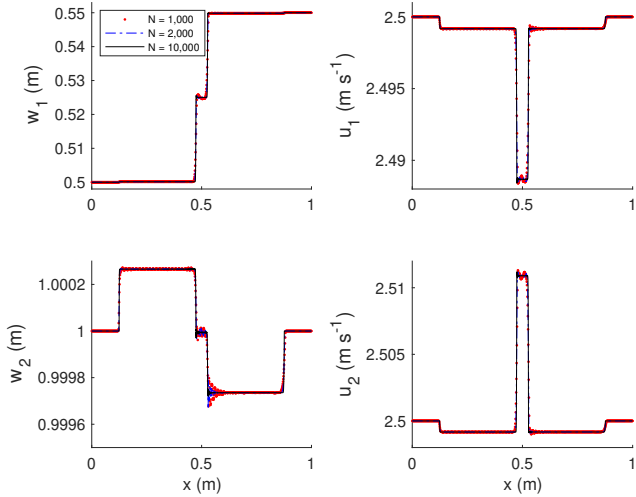}}
\caption{\label{fig:RiemannProblem} Solution to the Riemann problem at $t=0.12 $. The initial conditions are given by equation \eqref{eq:RiemannProblem}. Top left: internal layer $w_1$. Bottom left: external layer $w_2$. Top right: internal velocity $u_1$. Bottom right: external velocity $u_2$. In each panel, the reference solution is computed with a fine resolution using $10,000$ gridpoints (black solid line). The numerical approximations using $1000$ (red solid line) and $2000$ (blue dashed line) grid points are also shown.} 
\end{figure}

System \eqref{eq:swW2Hat} consists of balance laws with non-conservative products. This adds both theoretical and numerical challenges such as in calculating physically relevant weak solutions. Non-conservative products may change the jump conditions in weak solutions. In \cite{DalMasoetal1995}, a definition of weak solutions based on the theory of non-conservative products is provided and more on the theory of paths can be found in \cite{LeFloch2002,LeFloch2004}. 

In this example we test the convergence and non oscillatory properties of the proposed scheme. We consider a Riemann problem consisting of a two-layer flow along a channel with constant width, $\sigma(x,z) = 1$, and a flat bottom $B(x) = 0$. Friction and entrainment are ignored so that $n_i = n_b = 0$, and $V_e =0$, and the ratio of densities is $r = 0.98$. The initial conditions over the domain $x \in [0,1]$ consist of two constant states separated by a membrane at $x = 0.2$:

\begin{equation}
\label{eq:RiemannProblem}
(A_1,Q_1,A_2,Q_2)(x,0) = 
\left\{
\begin{array}{lr}
(0.5 ,1.25 ,0.5,1.25) & \text{ if } \ \ x \le 0.2 ,\\[0.07in]
(0.55,1.375,0.45,1.125) & \text{otherwise}.
\end{array}
\right.
\end{equation}
We impose the boundary conditions as specified at the beginning of Section \ref{sec:Results}. The two-layers move originally at the same speed, $u_1(x,0) = u_2(x,0) = 2.5$ in both sides of the membrane, and the only difference is a jump on the location of the interface between the two-layers at $t=0$ ($h_1$ increases and $h_2$ decreases by $0.05$ while the flow remains flat at the top).

Although the geometry of the channel is trivial and does not pose any challenges for the numerical scheme, when the membrane is removed, the initial conditions develop strong shock discontinuities. If solved with {\it non non-oscillatory} numerical schemes or over coarse grids, these discontinuities will lead to the onset of spurious oscillations. The numerical results shown in Figure \ref{fig:RiemannProblem} demonstrate the robustness of our scheme and shows the convergence to the discontinuous solution as the grid is refined. The left panels show the internal layer $w_1$ (top left) and the external layer $w_2$ (bottom left). The corresponding velocities are shown in the right panels. In each panel, different resolutions are used. The red dotted line shows the numerical results using $1,000$ gridpoints, while the blue dashed line represents the solution computed with a finer grid of $2,000$ gridpoints. Due to the presence of non-conservative products, a Riemann problem is particularly challenging in two-layer flows. Although one can observe oscillations when a resolution of $1,000$ gridpoints is used, such oscillations are significantly reduced for the reference solution that is obtained with a very fine resolution of 10,000 gridpoints. The structure of the solution can be clearly identified. There are four shockwaves, one of them with negative speed of propagation.

\subsection{Perturbation from a Steady State of Rest in a General Channel with Discontinuous Topography} \label{sec:ss_rest2}

A perturbation to a steady state of rest in a channel with general walls and discontinuous topography is applied in this section. The wall's width defined over the domain $[0,1]$ is given by
\begin{equation} \label{eq:channel}
\sigma(x,z) = 
\left\{
\begin{array}{ccl}
\frac{1}{2} + \half\sqrt{z} \left( 1-\frac{1}{4} \left( 1+\cos\left( \frac{\pi (x - 0.6)}{0.2} \right) \right) \right) & \text{ if } & 0.4 \le x < 0.8 \\[0.07in]
\frac{1}{2} \left( 1+\sqrt{z} \right) & \text{ if } & x \in [0,1] \setminus [0.4,0.8]. 
\end{array}
\right.
\end{equation}
The ratio of densities between layers is $r = 0.98$. Here $V_e = 0$ but friction is included with Manning coefficients $n_i = n_b = 0.009 \text{ s m}^{-1/3} $. Such values were obtained from \cite{khan2014modeling}.  The topography is given by
\[
B(x) = 
\left\{
\begin{array}{ccl}
0 & \text{ if } & x \in [0,0.15],\\
\frac{1}{4} \left( 1+\cos\left( 4 \pi (x - 0.4) \right) \right) & \text{ if } & x \in [0.15,0.4], \\[0.07in]
\frac{1}{4} & \text{ if } & x \in [0.4,1].
\end{array}
\right.
\]

\begin{figure}[h!]
\center{
\includegraphics[scale=0.45]{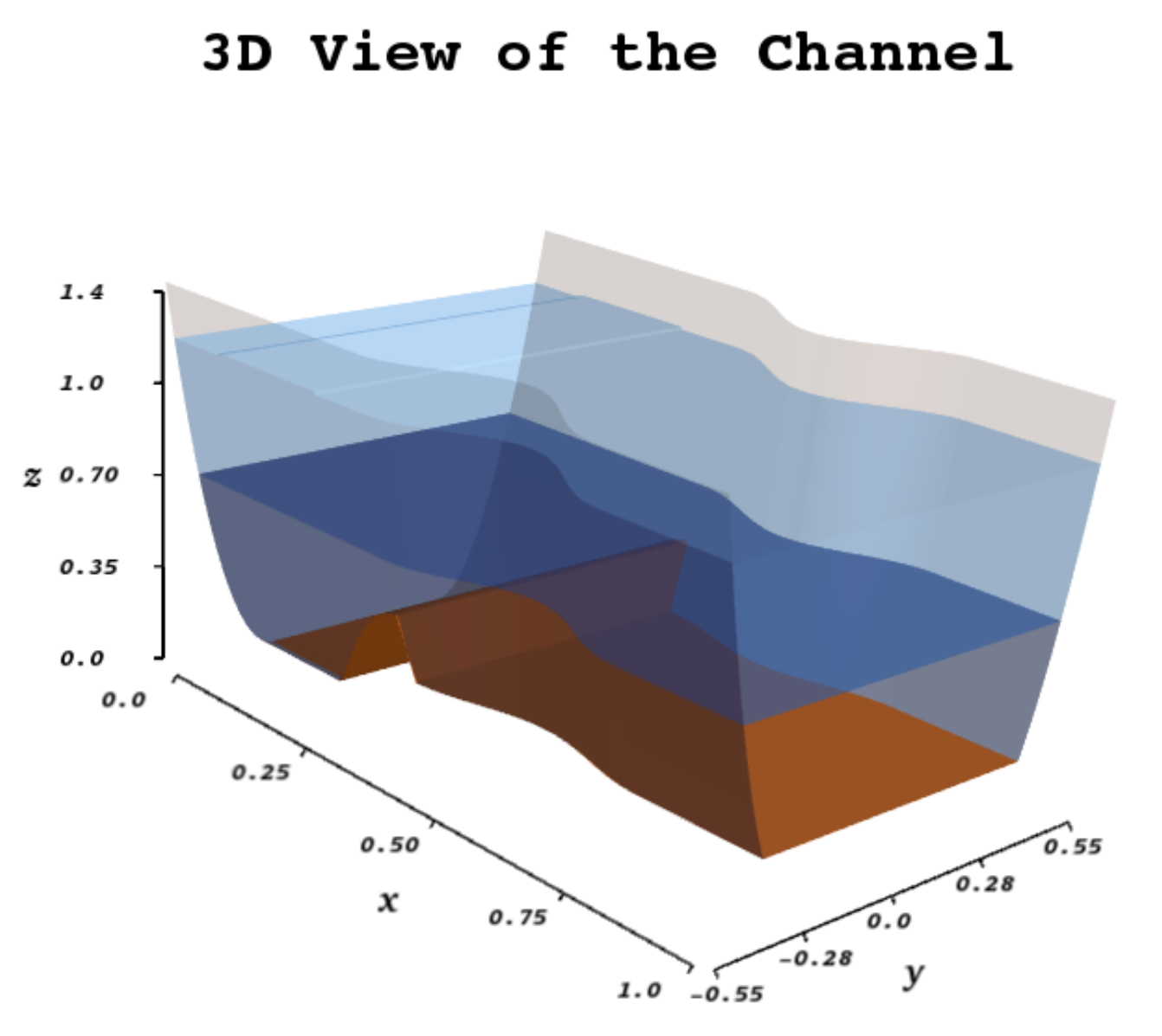}
\includegraphics[scale=0.4]{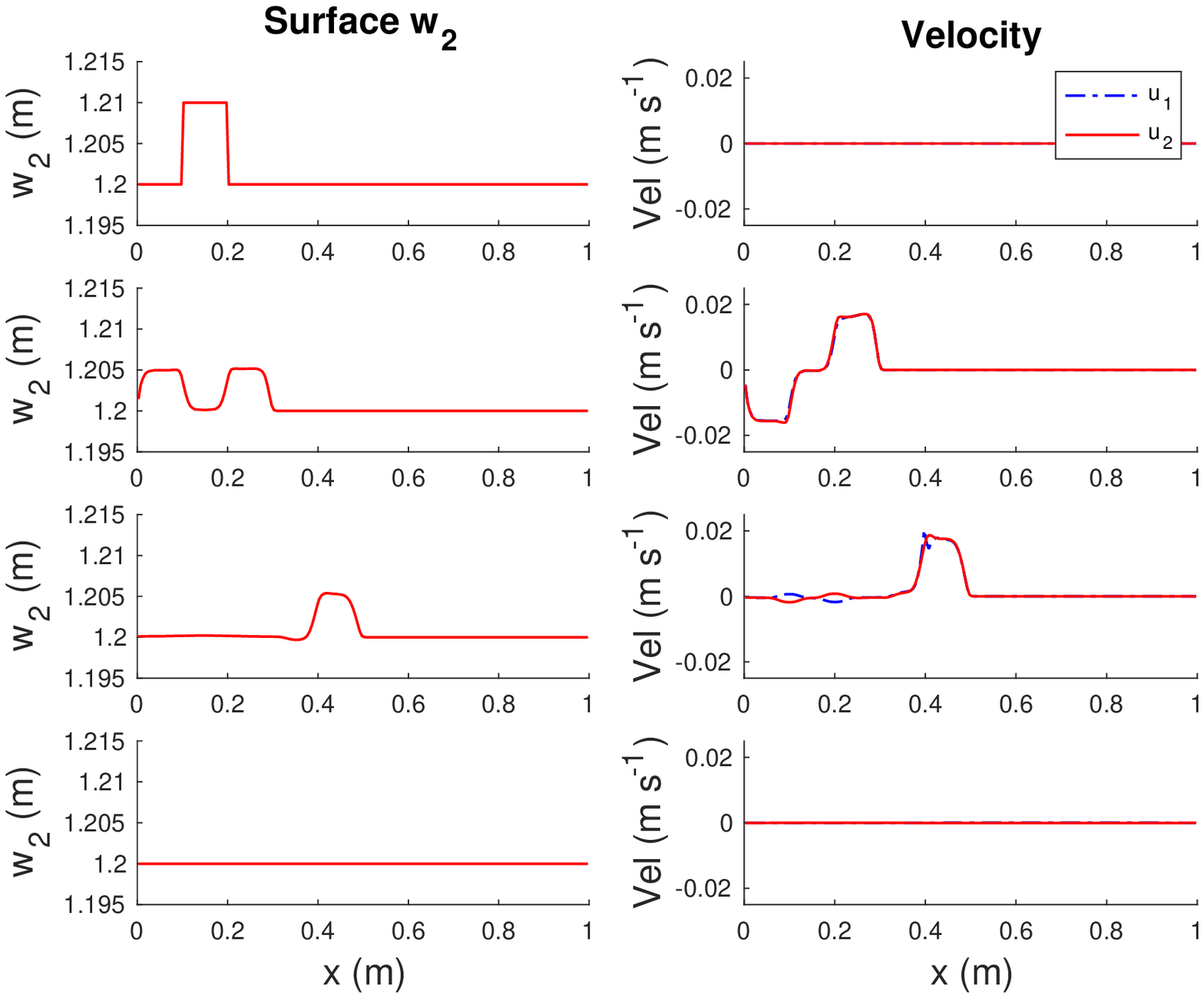}}
\caption{\label{fig:Pert2}Perturbation of steady state at rest. The initial conditions are given by \eqref{eq:rest2}. Left panel: 3D view of the channel at $t=0$. Middel panel: surface layer $w_2 = B+h_1+h_2$ at times $t= 0, 0.03, 0.1, 5$. Right panel: velocities $u_1$ and $u_2$ at times $t= 0, 0.03, 0.1, 5$. } 
\end{figure}

The initial conditions are given by
\beq \label{eq:rest2}
\begin{array}{l}
u_1(x,0) = u_2(x,0) = 0, \ \ \ \ w_1(x,0) =  0.7, \ \ \ \text{and} \ \ \
w_2(x,0) = 
\left\{
\begin{array}{ll}
1.2 + 10^{-2} & \text{ if } 0.1 \le x \le 0.2, \\[0.07in]
1.2 & \text{otherwise}.
\end{array}
\right.
\end{array}
\eeq

The left panel of Figure \ref{fig:Pert2} shows the 3D view of the channel described by equation \eqref{eq:channel} along with the initial conditions given by equation \eqref{eq:rest2}. Here we impose the boundary conditions as specified at the beginning of Section \ref{sec:Results}. The middle panel shows the external layer ($w_2$) at times $t=0,0.03,0.1,5$ in descending order and using a resolution of 200 grid points. The internal layer $w_1$ is not shown in the middle panels to identify the perturbation more clearly. The right panels show the velocity of each layer. Initially those velocities are zero. The densities in both layers are similar. The evolution of the flow is similar to that of a one-layer flow. This can be corroborated in the velocity plots where $u_1$ and $u_2$ are close to each other over time. Both velocities at $t=5$ are less than $0.83 \times 10^{-4}$ in the domain and the difference satisfies $||u_2-u_1||_\infty \le 2.9 \times 10^{-3}$.

\begin{figure}[h!]
\center{
\includegraphics[scale=0.28]{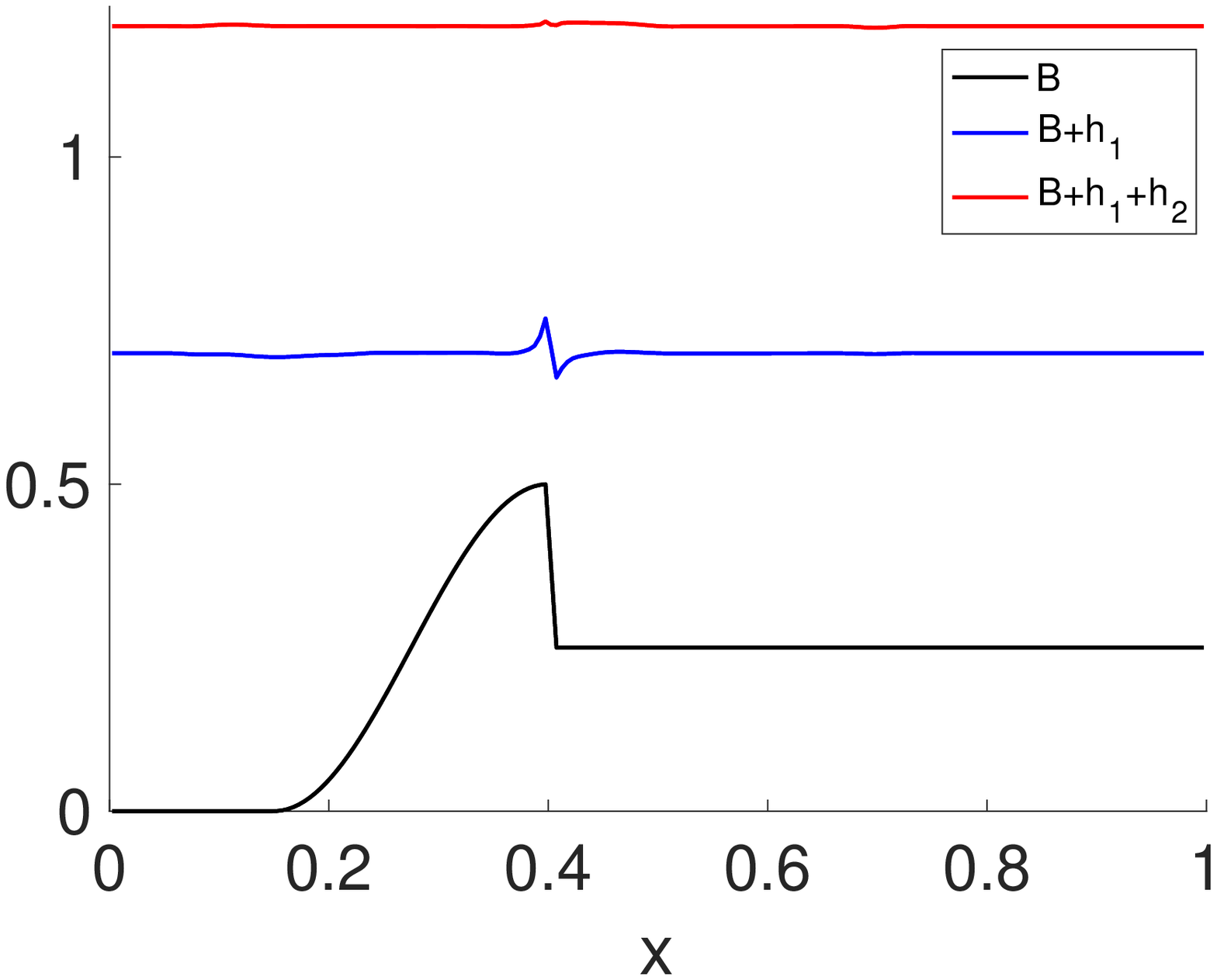}
\includegraphics[scale=0.65]{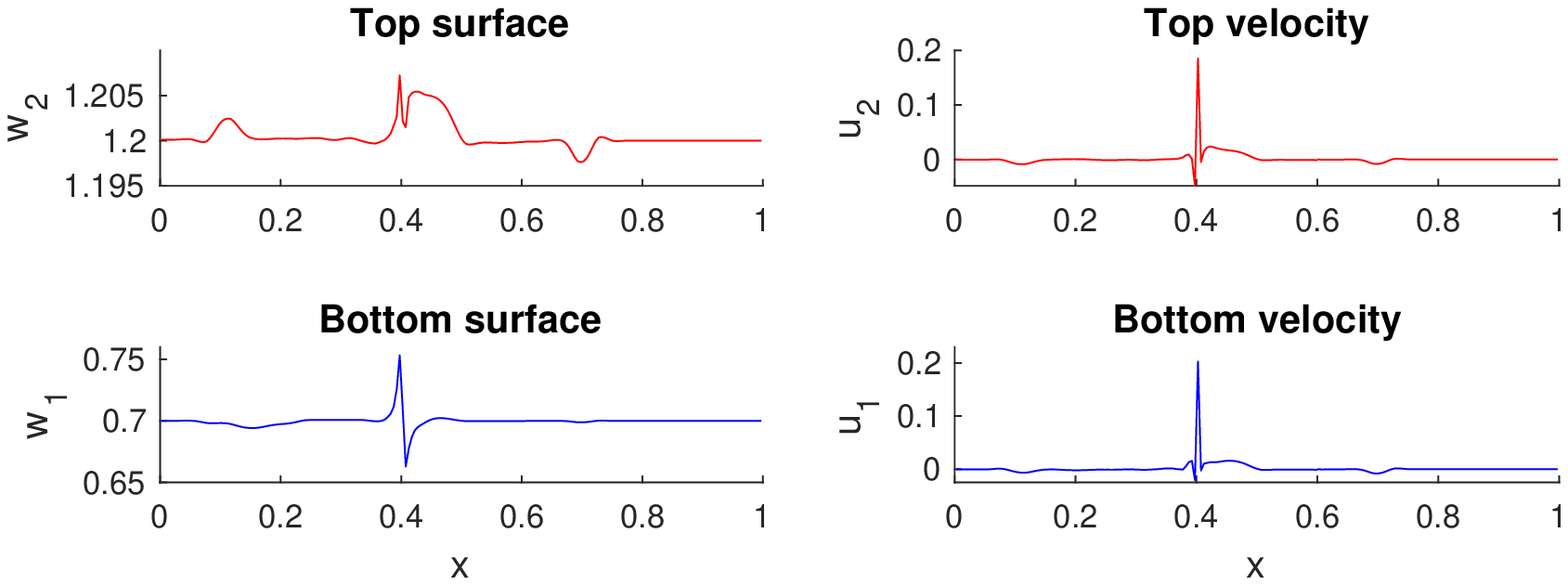}}
\caption{\label{fig:PertNonWB} The numerical results are shown when the numerical scheme does not satisfy the well-balance property. The initial conditions are those given in Figure \ref{fig:Pert2} without the perturbation. The total height $w_2$ is shown in the middle top panel at time $t=0.1$, while $w_1$, $u_2$ and $u_1$ are displayed in the middle bottom, right top and right bottom panels. Left panel: Topography (black solid line),  $w_1$ (blue solid line), and $w_2$ (red solid line) are shown. } 
\end{figure}

The well-balance property is particularly important in two-layer flows. Any unbalance in the external layer affects the internal one and viceversa. Figure \ref{fig:PertNonWB} (without the initial perturbation) shows the numerical results when the numerical scheme does not satisfy the well-balance property. In particular, here we use the same scheme but implement the reconstruction of $A_1$ and $A_2$ directly, instead of the procedure described in Section \ref{sec:WellBalance}. In a channel with width variations like the one considered in this section, the reconstruction results in a top surface $w_2$ that is not flat already in the first time step. The time evolution of $w_2,w_1,u_2$ and $u_1$ are shown in the middle top, middle bottom, right top and right bottom panels respectively. All those quantities must be flat in this steady state at rest but the unbalance in the numerical scheme caused numerical errors, which are significant as one can see in the left panel that shows the topography and the total heights of each layer.

\subsection{Internal Waves}

\begin{figure}[h!]
\center{
\includegraphics[width=0.38\textwidth]{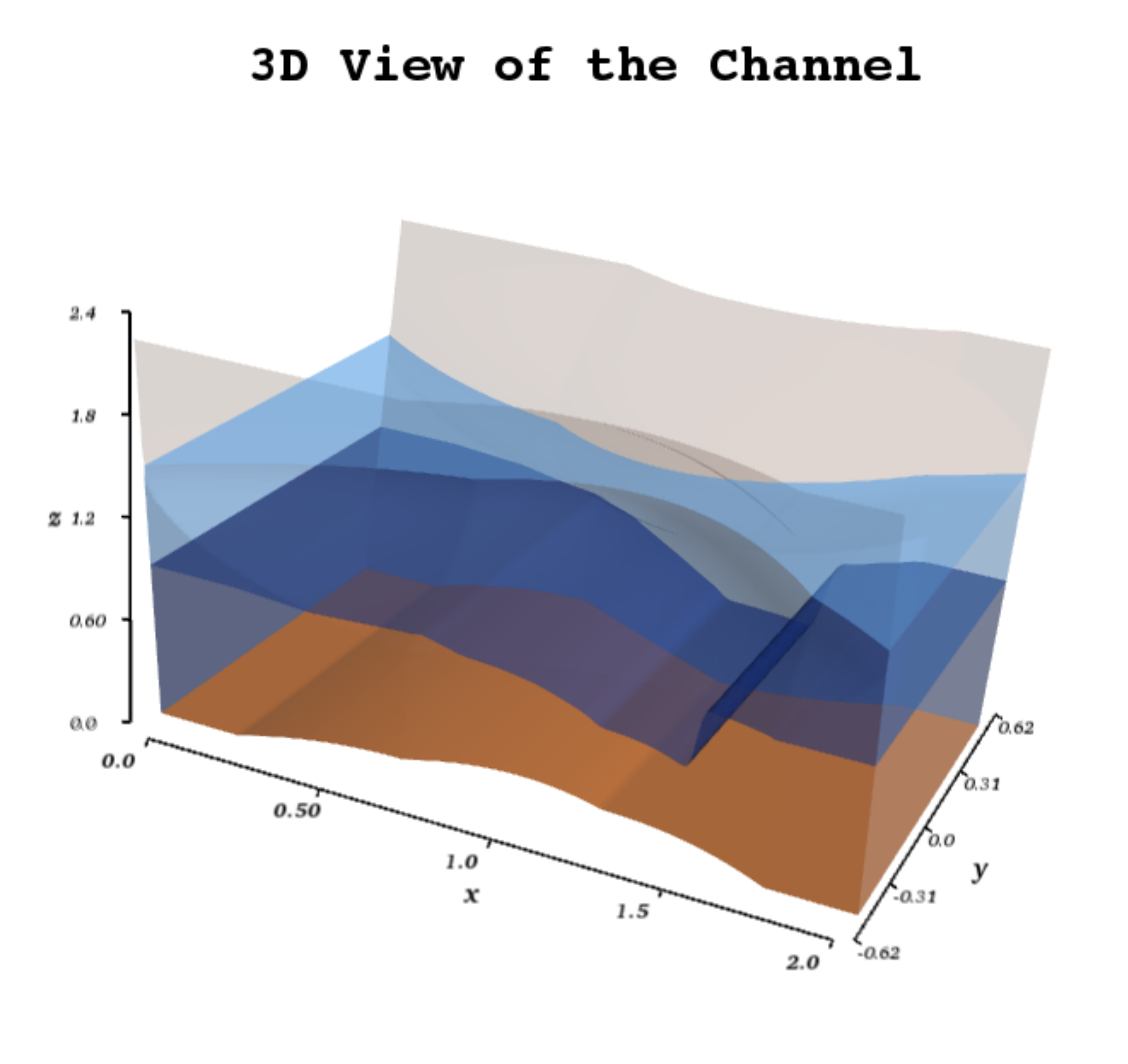}
\includegraphics[width=0.42\textwidth]{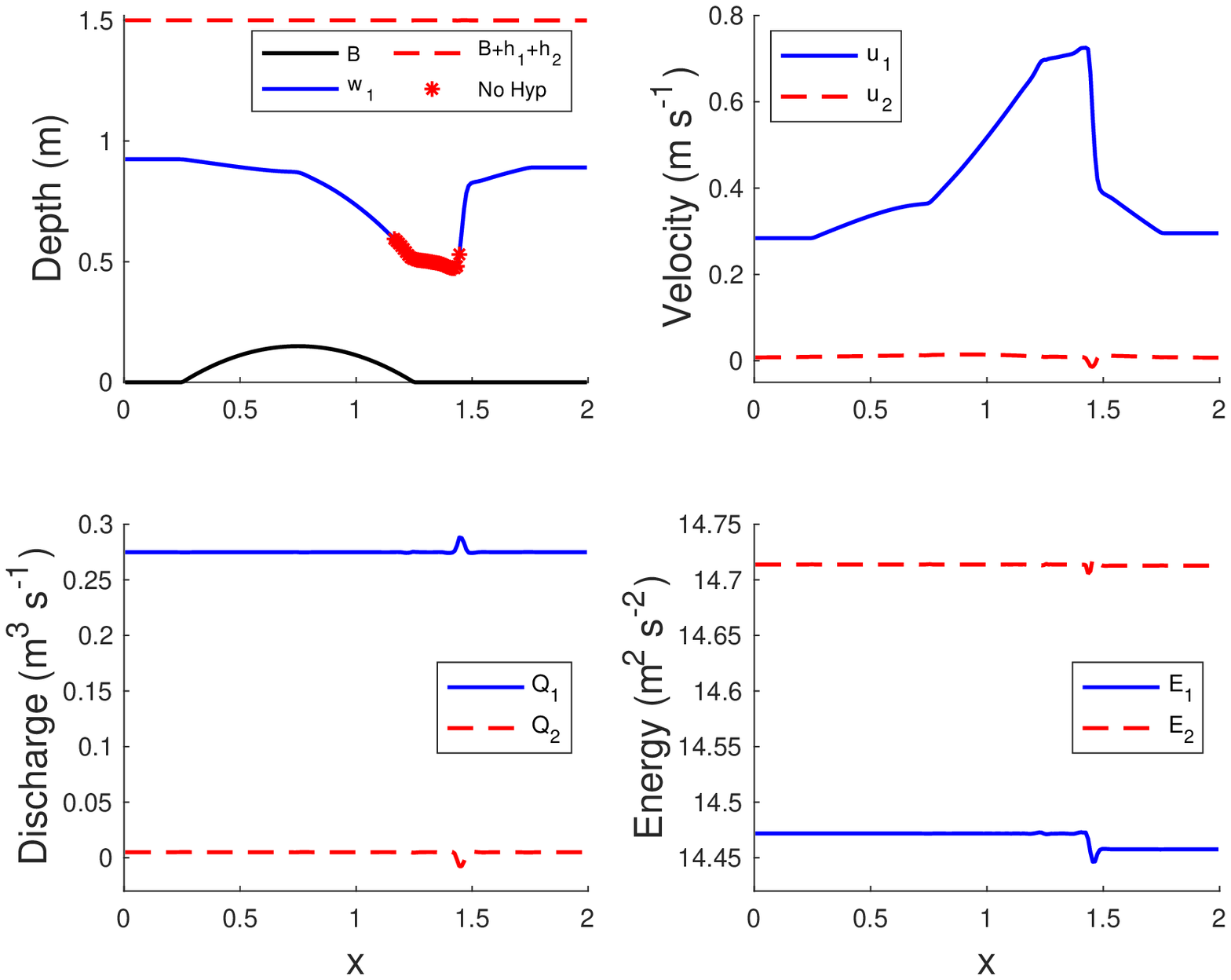}}
\caption{\label{fig:InternalWavesSS} Internal waves. Left panel: 3D views of a steady state where the internal layer is not at rest. Middle top panel: plots of $w_1$ (blue solid line), $w_2$ (red dashed line), and topography (black solid line). The red asterisks show locations where hyperbolicity is lost. The top right panel shows the external (red dashed line) and internal (blue solid line) velocities. The middle bottom and bottom right panels repeat the plots for $Q_1$, $Q_2$ (middle bottom) and $E_1$, $E_2$ (bottom right). } 
\end{figure}

Waves due to gravity may form in the internal layer, generating internal waves in two-layer flows. Internal waves may appear in rivers, and shear flow instabilities may be one of the mechanisms leading to such waves \cite{apel1975observations}. Here we test the ability of the numerical scheme to capture steady states that are not at rest for the internal layer while the free surface is at rest. We consider a channel extended over the interval $x\in [0,2]$ with the topography and channel's width given, respectively, by
\begin{equation}
\label{eq:BPositivity}
B(x) = 0.3 \max\left( \frac{1}{2}-2 \left( x-\frac{3}{4} \right)^2,0\right), 
\end{equation}
and
\begin{equation}
\label{eq:sigmaPositivity}
\sigma(x,z) = 1-\frac{1}{2} \max\left( \frac{1}{2} - 2 \left( x- 1.25 \right)^2,0 \right)+\frac{z}{10}-\frac{3}{2} \max\left( \frac{1}{2}- \frac{1}{2} (x-1)^2-\left( z-\frac{3}{2} \right)^2,0\right).
\end{equation}
For the purpose of this numerical test, we ignore entrainment ($V_e=0$). We also ignore friction ($n_i = n_b = 0$) to verify that the conditions for steady states in Proposition \eqref{prop:SS} are met.

The discretization of the source terms \eqref{eq:SDisc} guarantees the balance between source terms and non linear fluxes for steady states of rest, but that balance is not guaranteed in steady states characterized by the conditions \eqref{eq:InternalWave}, where the equilibrium in the internal layer is achieved through the effects of hydrostatic pressure exerted by the external layer. Perturbations in any of the layers affect each other. The results obtained for this example demonstrate that the proposed numerical schemes computes correctly these non trivial steady states. 

It is also worth noting that in these examples, despite the loss of hyperbolicity in a subregion, the approximations \eqref{eq:EigBounds} that we employ to bound the eigenvalues allow us to capture the steady-state flows and resolve accurately their internal waves. Other steady states with large enough velocity $u_1$, however, may lose hyperbolicity in the entire domain. In such cases, the proposed model is not suitable to describe those type of flows and should not be used at all.

Figure \ref{fig:InternalWavesSS} shows one example of internal waves. The steady state was obtained numerically by running the model for a long time, i.e., by convergence. The boundary conditions in this case are as follows. At inflow, we specify $w_1 = 0.9, u_1 = 0.3, w_2 = 1.5, u_2 = 0$, and those same values are used for the initial conditions. We confirm the solution is a steady state with the plots of $Q_1,Q_2,E_1,E_2$ in the bottom middle and bottom right panels. The discharges $Q_1,Q_2$ and the energy $E_1$ are approximately constant. The energy $E_2$ seems to be converging to a piecewise constant profile. The topography, and the elevations of the internal and external layers are shown in the middle top panel, while the velocity is displayed in the top right panel. We clearly see that $w_2$ is constant, indicating that the external layer is at rest. The 3D view of this flow is shown in the left panel. In the middle top panel we also show with red asterisks the region where hyperbolicity is lost. We notice that hyperbolicity is lost where $u_1$ is highest. This is consistent with the fact that hyperbolicity is lost when $u_1$ and $u_2$ are far from each other. 

\begin{figure}[h!]
\center{
\includegraphics[width=0.42\textwidth]{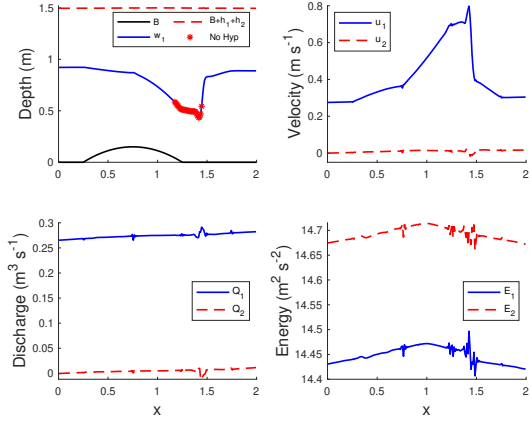}}
\caption{\label{fig:InternalWavePertNonWB} Perturbation evolution of internal wave. As in Figure \ref{fig:InternalWavesSS}, the solution at $t=0.2$ is computed here with the same central-upwind scheme without the well-balance property. Top left panel: topography (black solid line), $w_1$ (blue solid line) and $w_2$ (red dashed line). Top right panel: Velocities $u_1$ (blue solid line) and $u_2$ (red dashed line). The discharge and energy of each layer are shown in the bottom left and bottom right panels respectively.} 
\end{figure}

We now continue the simulation of the steady state reached in Figure \ref{fig:InternalWavesSS} with a reconstruction that does not follow the procedure in sections \ref{sec:reconstruction} and \ref{sec:WellBalance} where $w_1$ and $w_2$ are reconstructed frist. Instead, both $A_1$ and $A_2$ are reconstructed directly with the non-oscillatory procedure in Section \ref{sec:reconstruction}. As a consequence, this version of the numerical scheme does not satisfy the well-balance property. The numerical results are shown in Figure \ref{fig:InternalWavePertNonWB} at time $t=0.2$, where the initial conditions correspond to the steady state reached in Figure \ref{fig:InternalWavesSS}. The two heights are shown in the top left panel, while the velocities are shown in the top right panel. Here we expect the solution to be near the original steady state as in Figure \ref{fig:InternalWavesSS} but observe large numerical errors instead. For instance, the discharge (bottom left) and energy  (bottom right) show large oscillations and are not converging to constant or piece-wise constant values anymore. The well-balance property is particularly important in two-layer shallow-water flows. 

\begin{figure}[h!]
\center{
\includegraphics[width=0.35\textwidth]{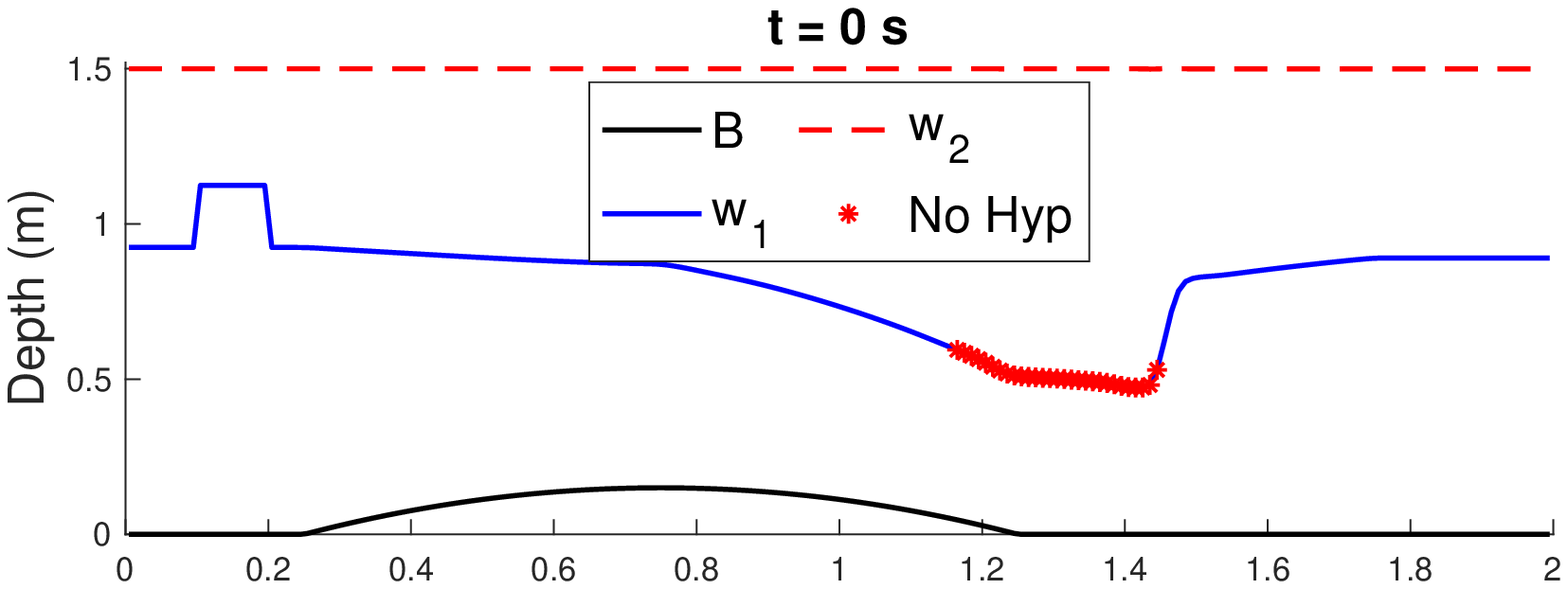}
\includegraphics[width=0.35\textwidth]{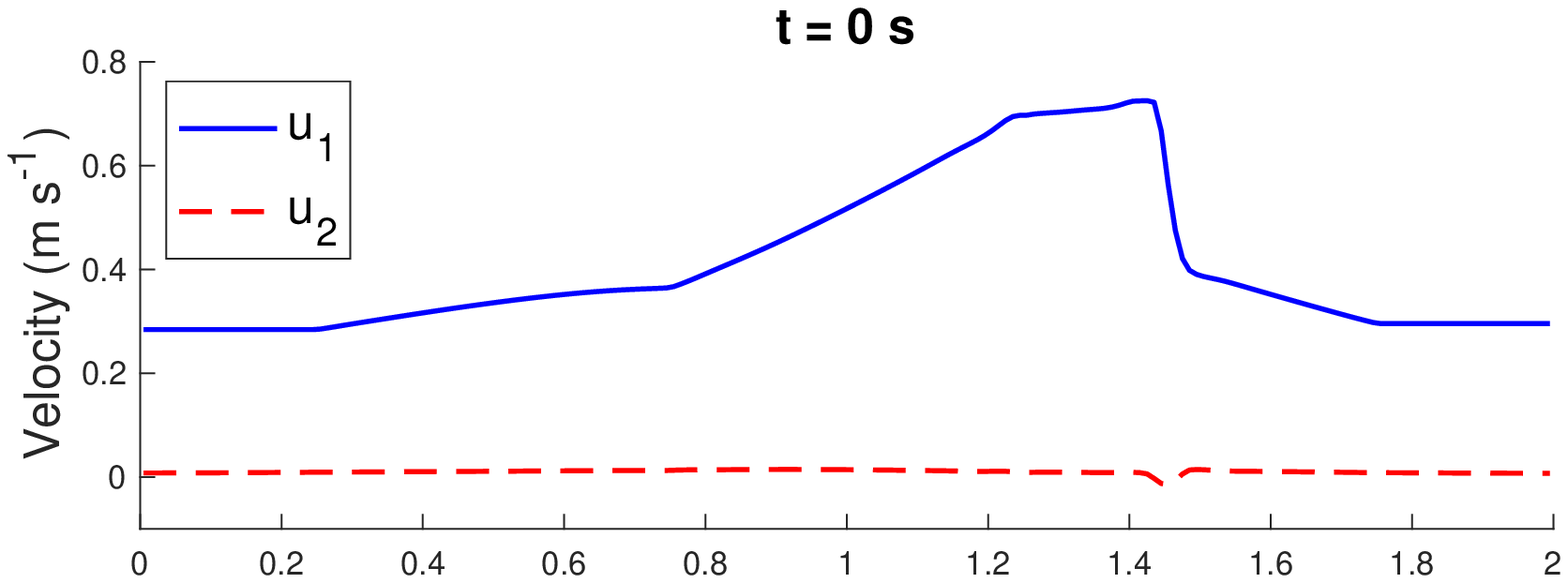}

\includegraphics[width=0.35\textwidth]{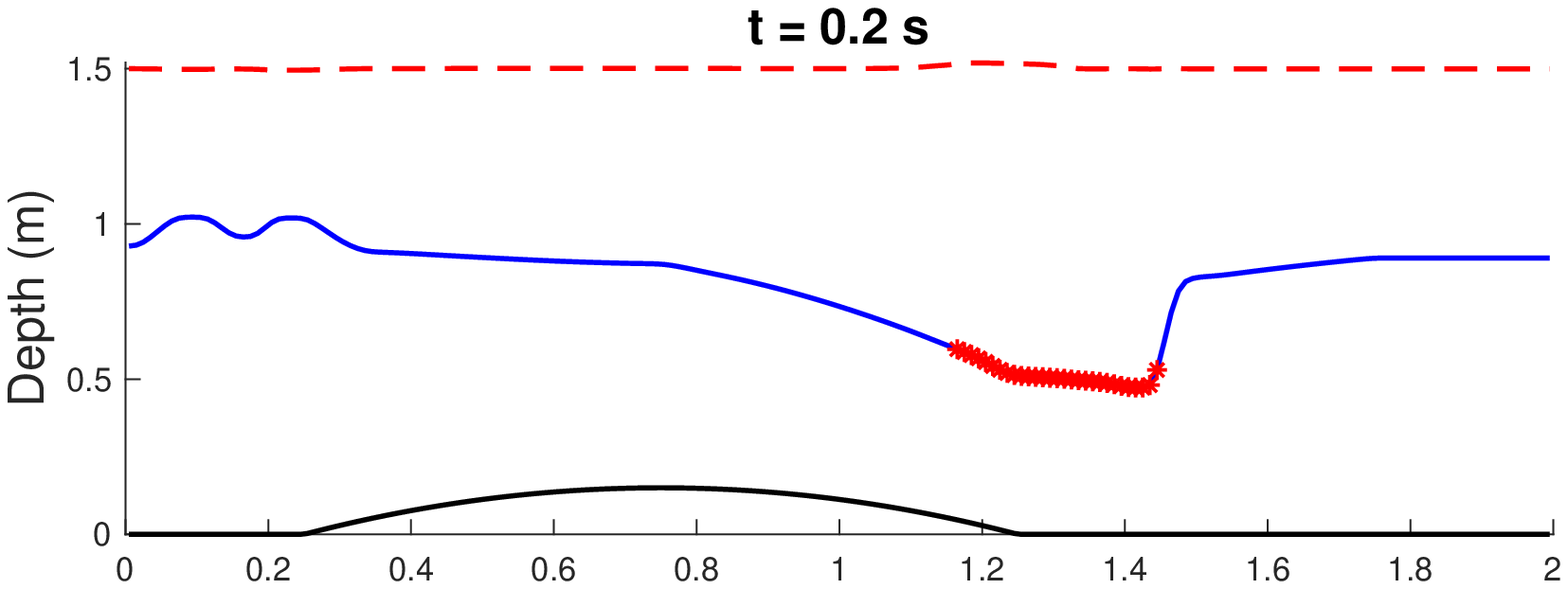}
\includegraphics[width=0.35\textwidth]{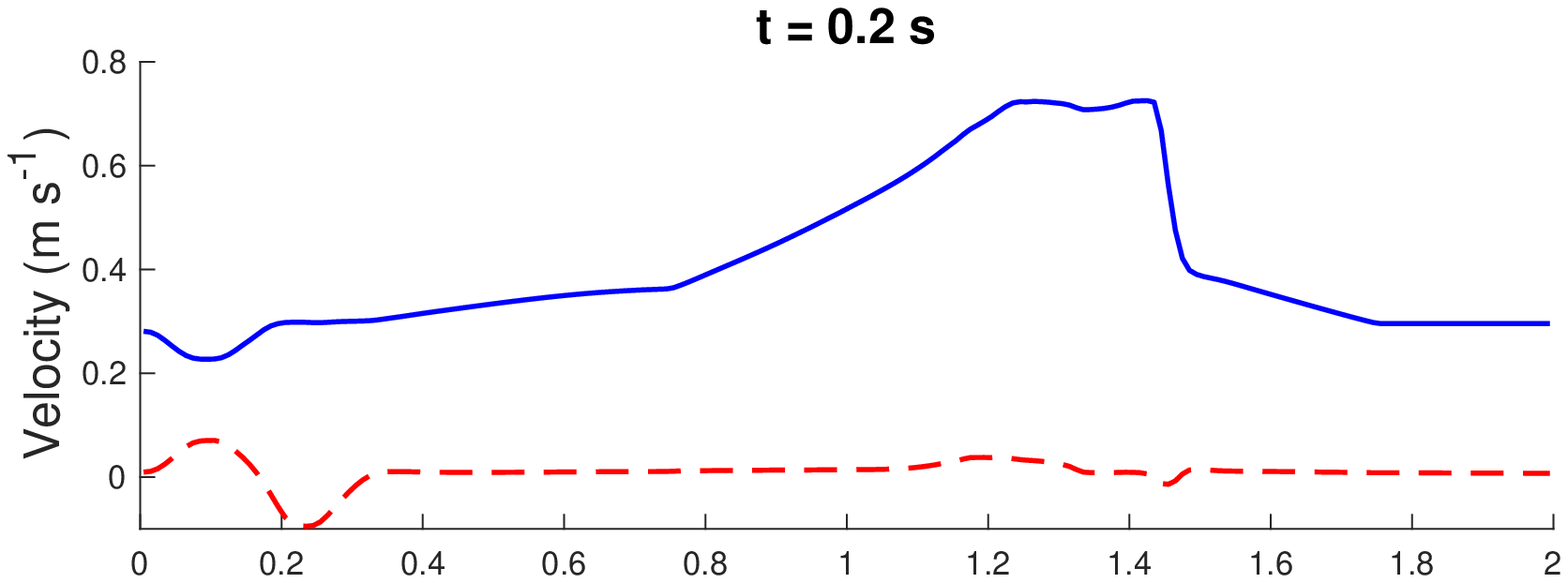}

\includegraphics[width=0.35\textwidth]{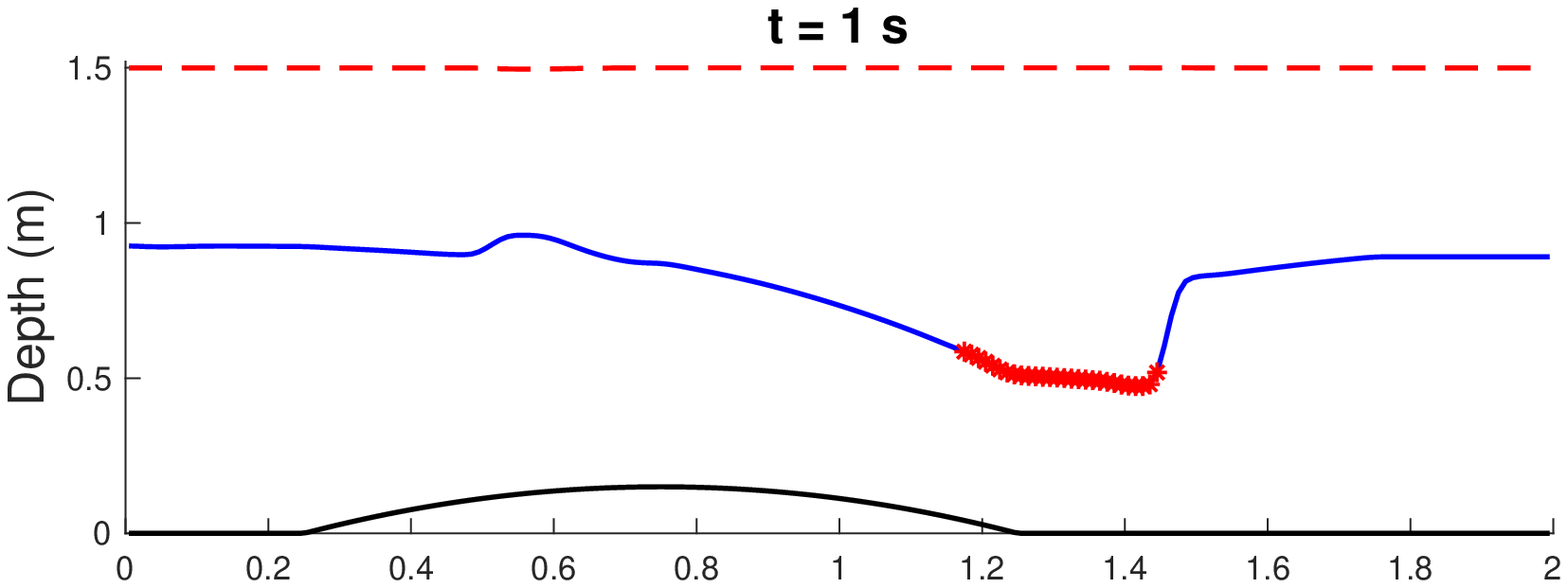}
\includegraphics[width=0.35\textwidth]{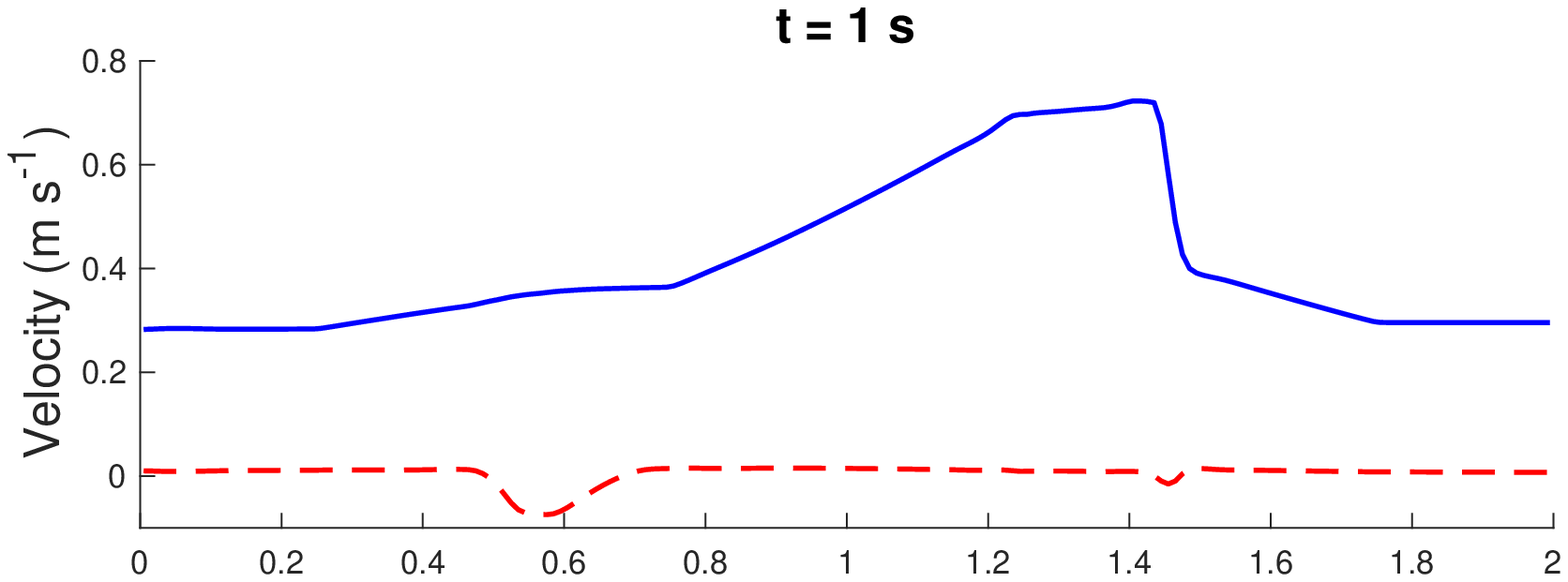}

\includegraphics[width=0.35\textwidth]{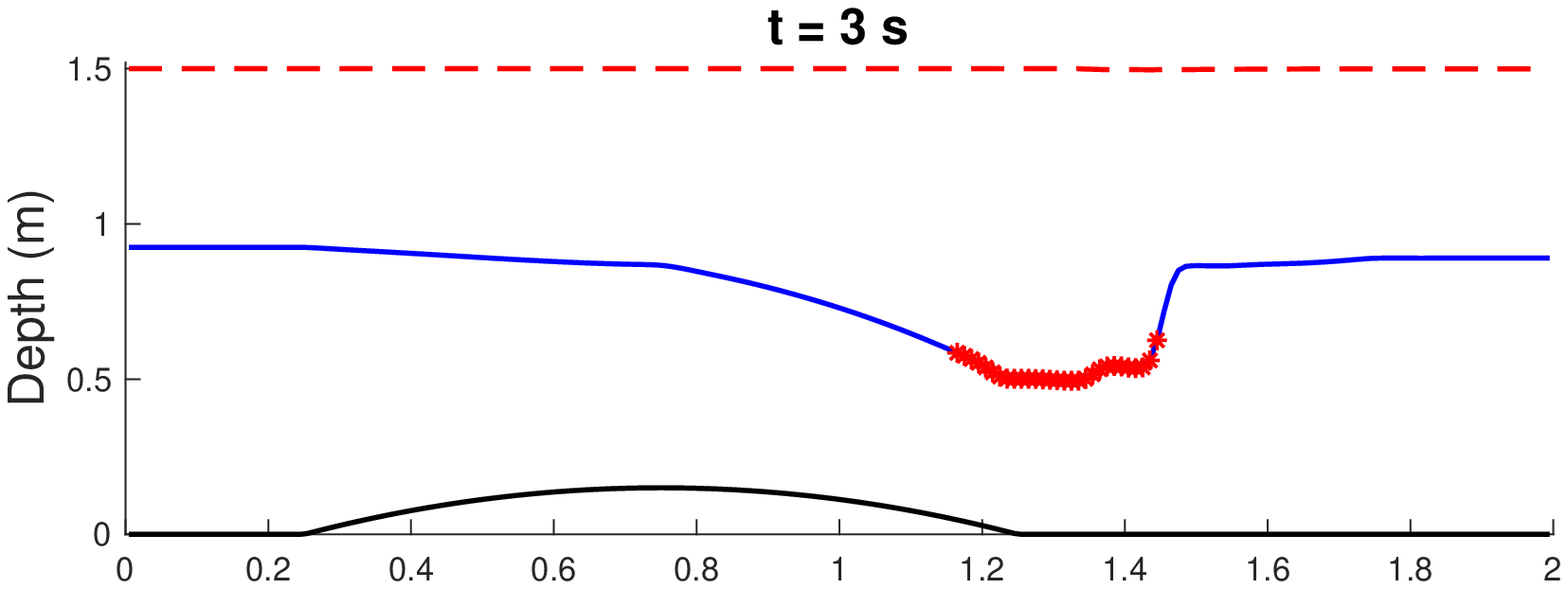}
\includegraphics[width=0.35\textwidth]{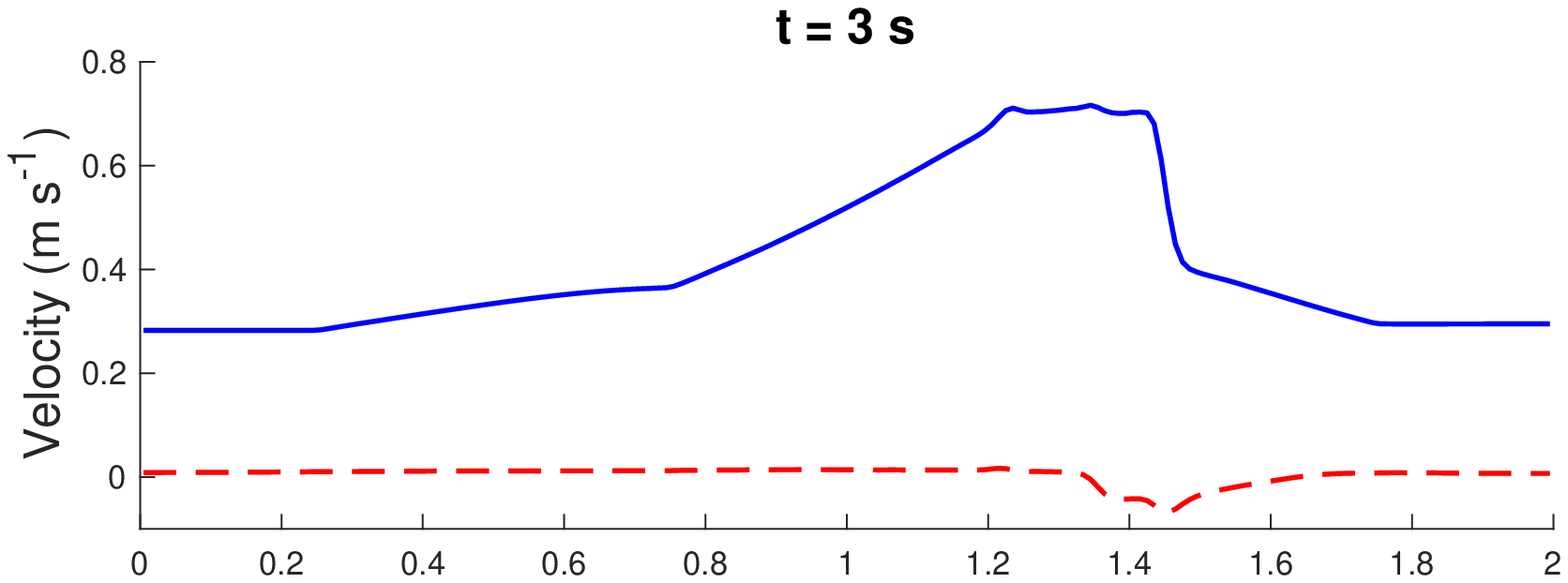}

\includegraphics[width=0.35\textwidth]{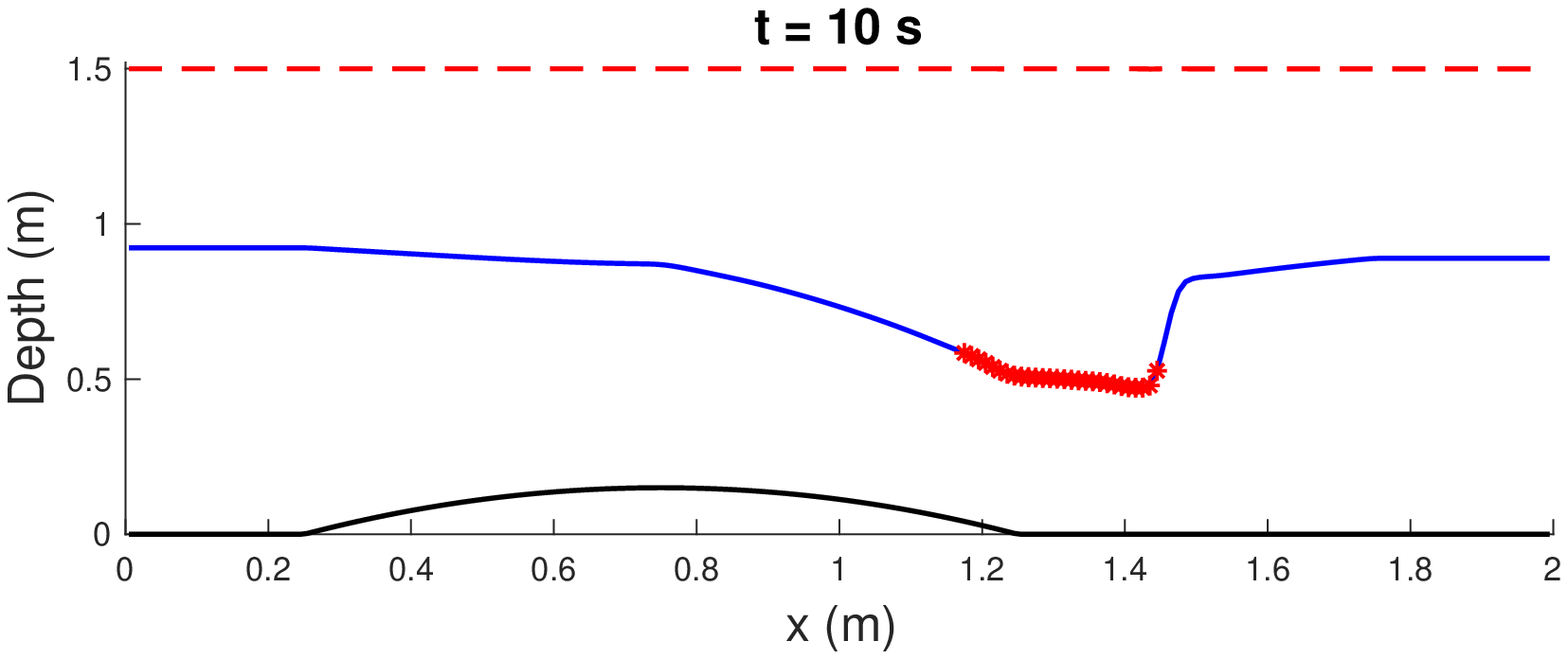}
\includegraphics[width=0.35\textwidth]{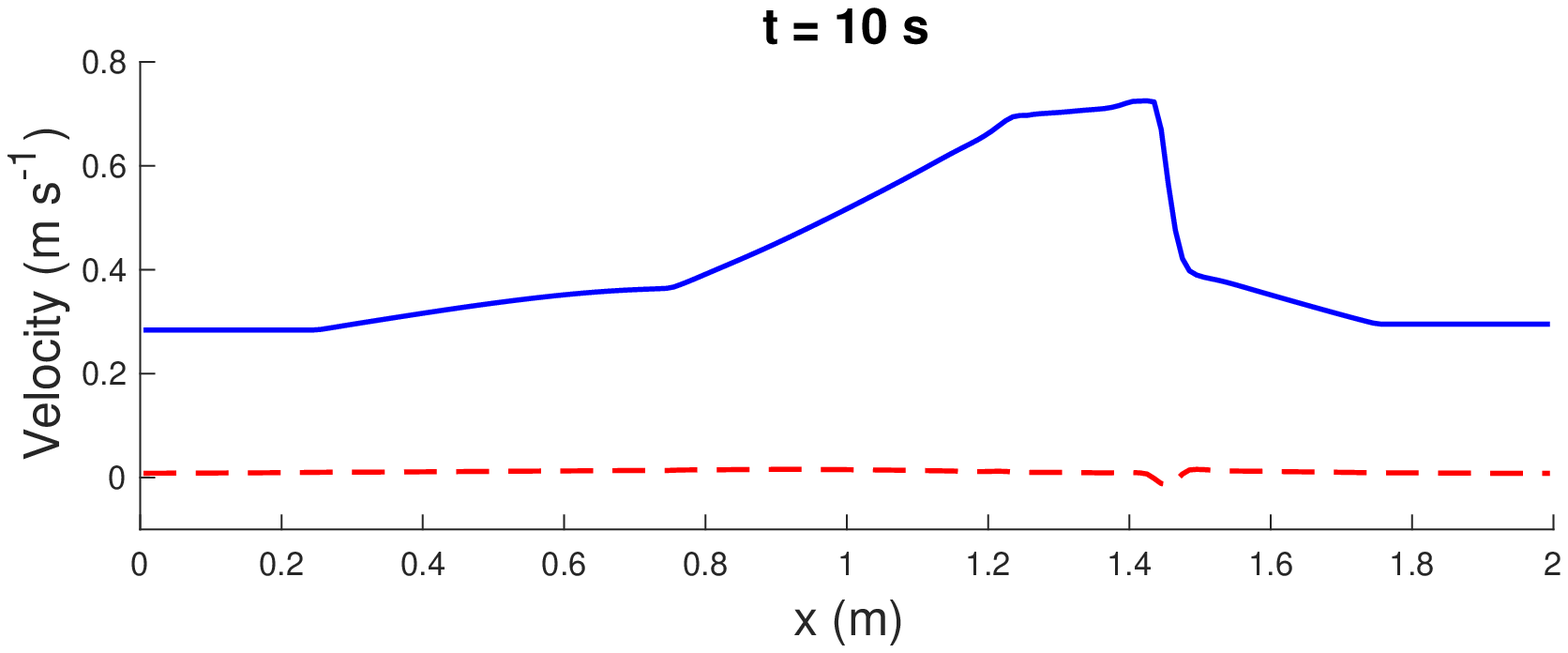}}
\caption{\label{fig:InternalWavePert} Perturbation evolution of internal wave. A perturbation of $\delta_w = 0.2$ is added to the bottom depth $w_1$ in the interval $[0.1,0.2]$. From top to bottom: the time evolution at times $t=0,0.2,1,3,10$ is shown for $w_1$ and $w_2$ in the left column, and $u_1$ and $u_2$ in the right column respectively. } 
\end{figure}

We further test the numerical scheme by adding a perturbation of $\delta_w = 0.2$ in the internal layer in the interval $[0.1,0.2]$. The time evolution of this perturbation at times $t=0,0.2,1,3,10$ is shown in Figure \ref{fig:InternalWavePert} from top to bottom for $w_1,w_2$ (left column) and $u_1,u_2$ (right column). We observe that the perturbation in the internal layer propagates throughout the domain and eventually leaves it, recovering its initial profile. On the other hand, the external layer $w_2$ is always almost flat. The formulation in equation \eqref{eq:swW2Hat} helps maintaining the balance in this internal wave flow and computing the free surface correctly. 

%

\subsection{Lock Exchange}
\label{sec:LockExchange}

\begin{figure}[h!]
\center{
\includegraphics[width=0.35\textwidth]{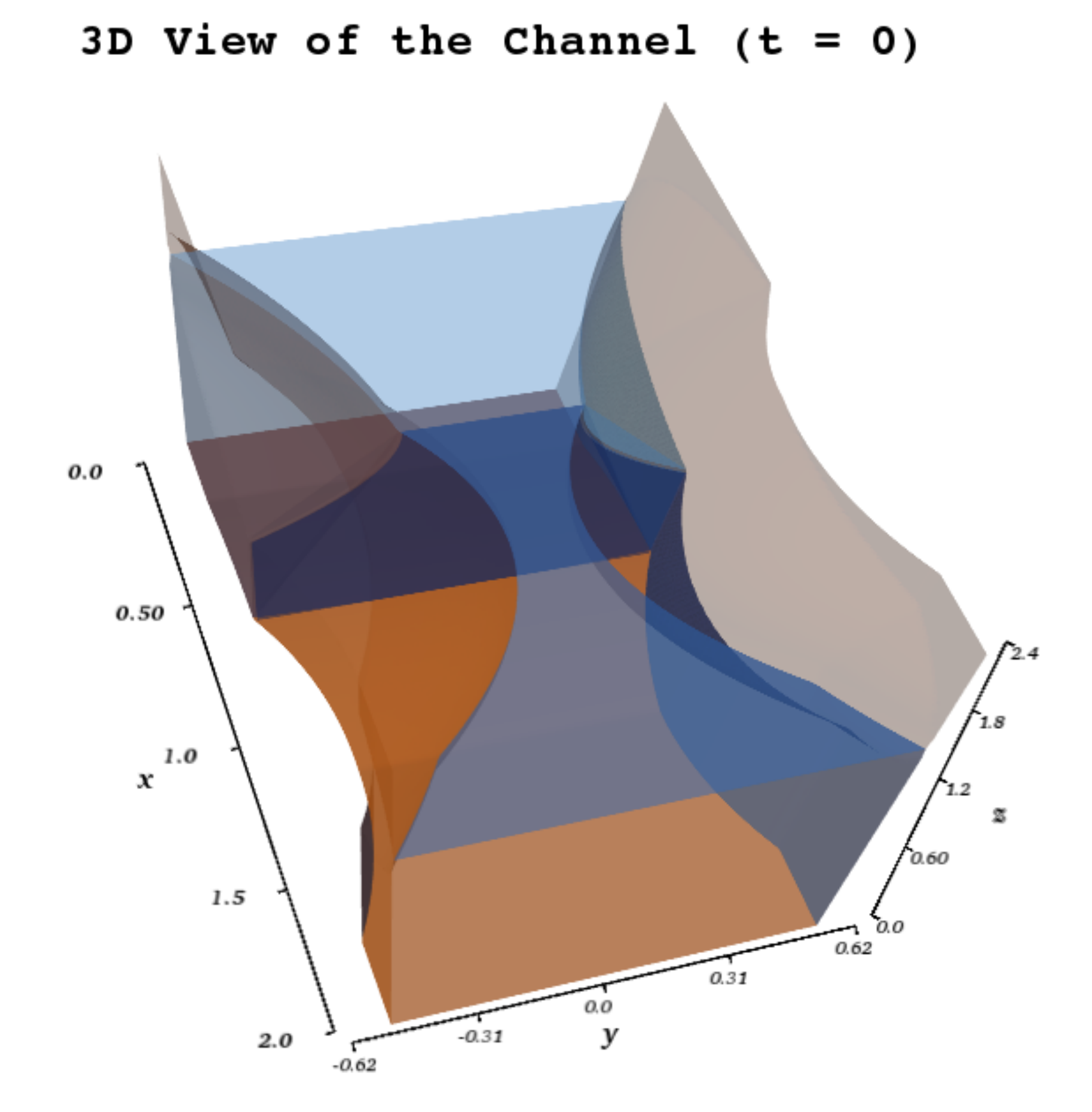} \ \ 
\includegraphics[width=0.49\textwidth]{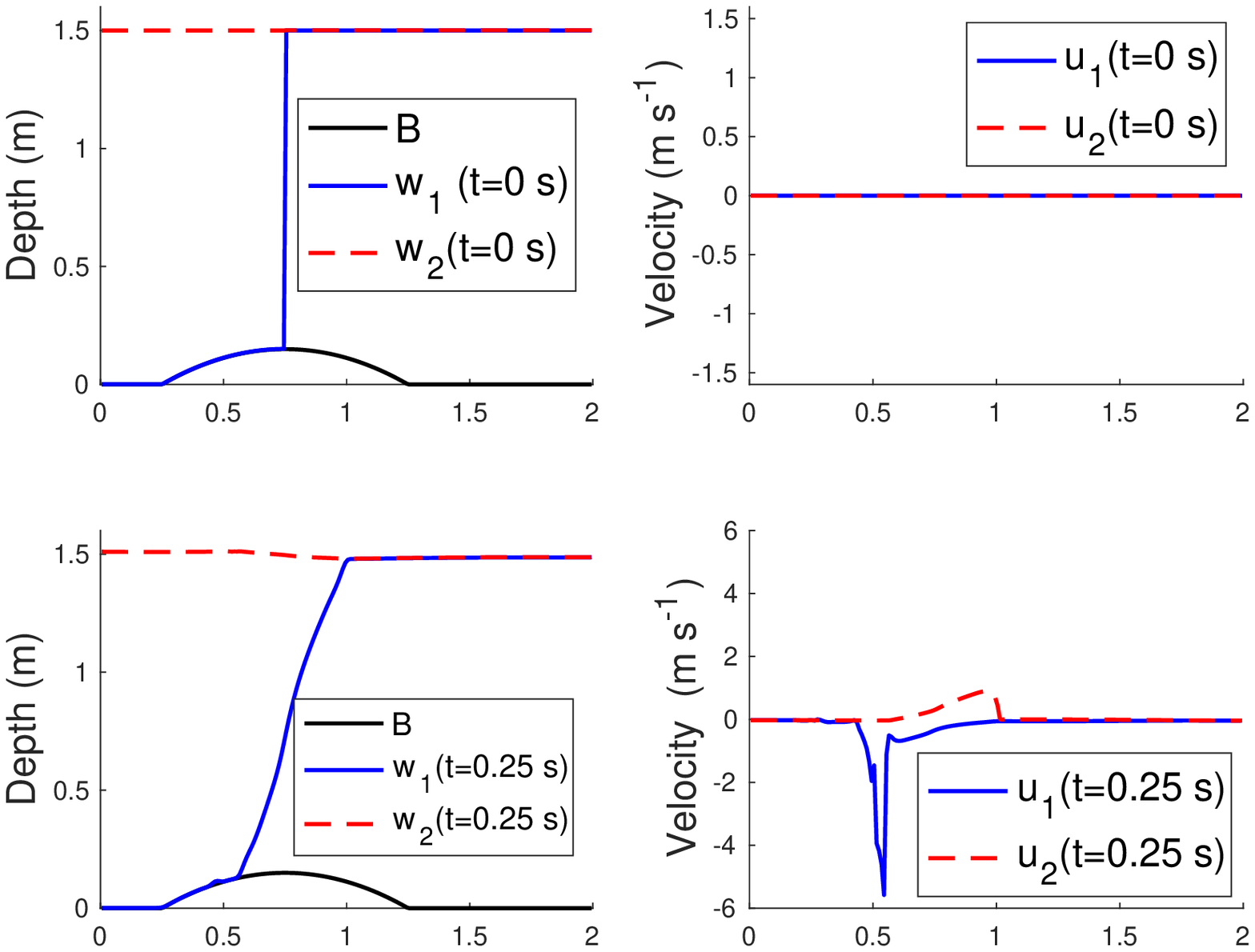} \\[0.07in]
\includegraphics[width=0.35\textwidth]{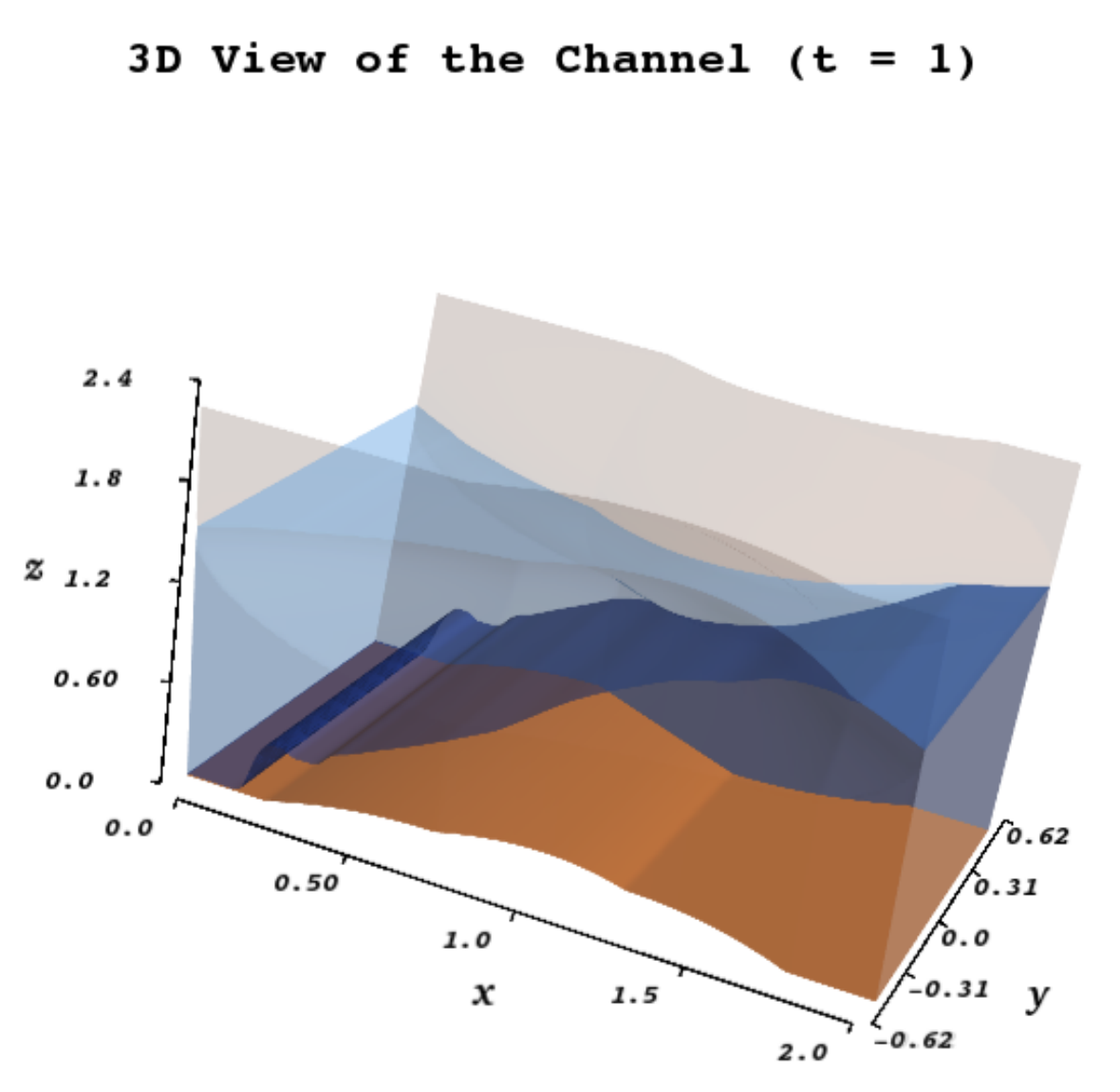} \ \ 
\includegraphics[width=0.49\textwidth]{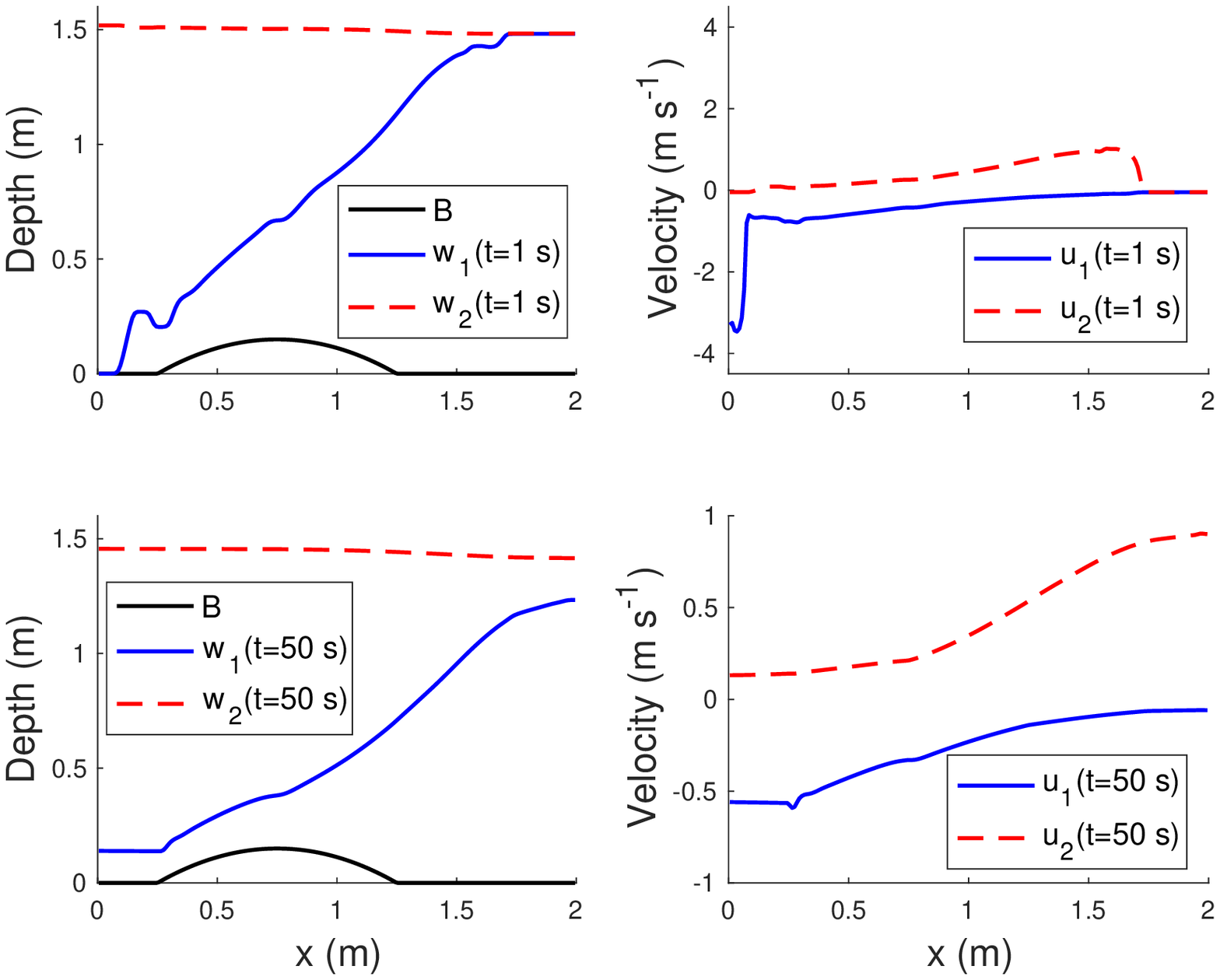}}
\caption{\label{fig:Positivity} Positivity numerical test with initial conditions given by \eqref{eq:InitCondPositivity}, topography given by \eqref{eq:BPositivity} and wall's width by \eqref{eq:sigmaPositivity}. Left column: 3D plot at times $t=0$ and $t=1$. Middle column: topography (black solid line), $w_1$ (blue solid line), and $w_2$ (red dashed line) at times $t=0,0.25,0.5,50$. Right column: the same as in the middle column for $u_1$ (blue solid line) and $u_2$ (red dashed line).} 
\end{figure}

In this example we test the positivity preserving property of the numerical scheme presented in Section \ref{sec:scheme}. The topography and wall's width are given by equations \eqref{eq:BPositivity} and \eqref{eq:sigmaPositivity} from the previous numerical test. The Manning coefficients for the interface and bottom are $n_i = n_b = 0.009 \text{ s m}^{-1/3} $, and $V_e=0$. 

The two-layers are initially at rest, $u_1(x,0) = u_2(x,0) = 0$, with the external layer occupying the left side of the channel and the internal layer the right side. That is,
\begin{equation}
\label{eq:InitCondPositivity}
w_1(x,0) = \left\{ \begin{array}{lcl} B(x)+\delta_B & \text{ if } & x \le 0.75,\\[0.07in]
1.5 & \ & \text{otherwise,} \end{array} \right. 
\ \ \ \ \ \text{ and } \ \ \ \ \ 
w_2(x,0) = \left\{ \begin{array}{lcl} 1.5  & \text{ if } & x \le 0.75,\\[0.07in]
1.5 + \delta_B & \ & \text{otherwise.} \end{array} \right.
\end{equation}
Here $\delta_B = 10^{-3}$ is used as a threshold. We note that this threshold is used to avoid loss of hyperbolicity. The ratio of densities is $r=0.95$. The boundary conditions used here are those specified at the beginning of Section \ref{sec:Results}. We expect the heavier fluid to push the external layer, resulting in a displacement of the internal layer below the external one. The time evolution at times $t = 0, 0.25, 1, 50$ are displayed in Figure \ref{fig:Positivity}. The elevations are shown in the middle column. As one can observe, preserving positivity is challenging and such property in the numerical scheme adds stability and accuracy to the approximated numerical solution. At $t=50$, the flow has reached a steady state that is not at rest. The velocity in the right column indicates that the external layer moves to the right while the internal one goes in the opposite direction. The 3D view on the left column exhibits the complex geometry of the channel's wall and the fluid's evolution over time.

\subsection{Currents}
\label{sec:CurrentsEntrainment}

\begin{figure}[h!]
\center{
\includegraphics[width=0.42\textwidth]{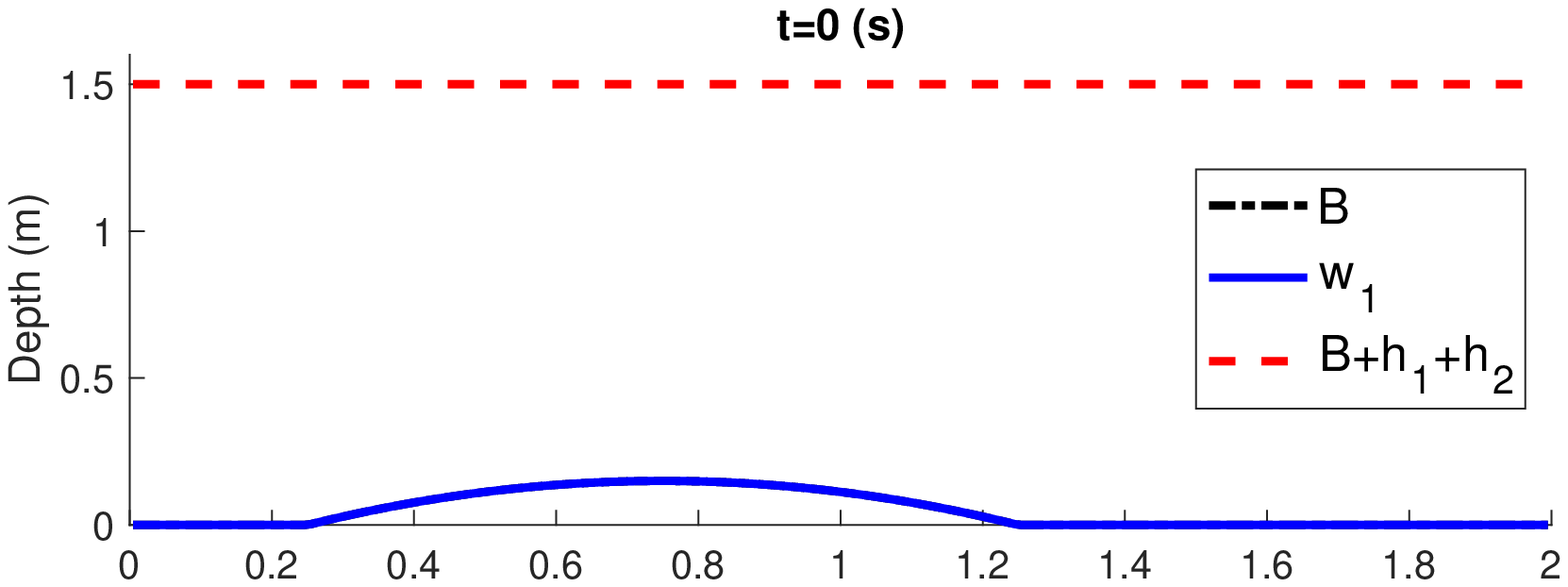} 
\includegraphics[width=0.4\textwidth]{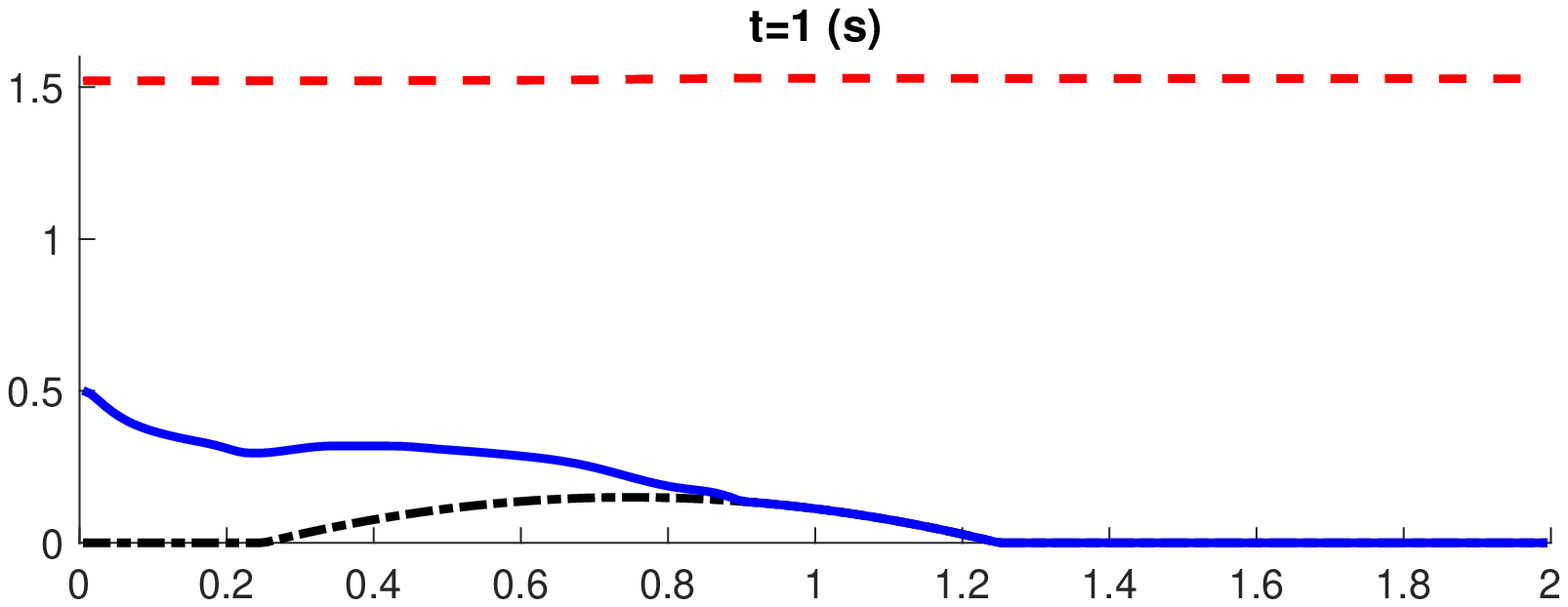} 
\includegraphics[width=0.4\textwidth]{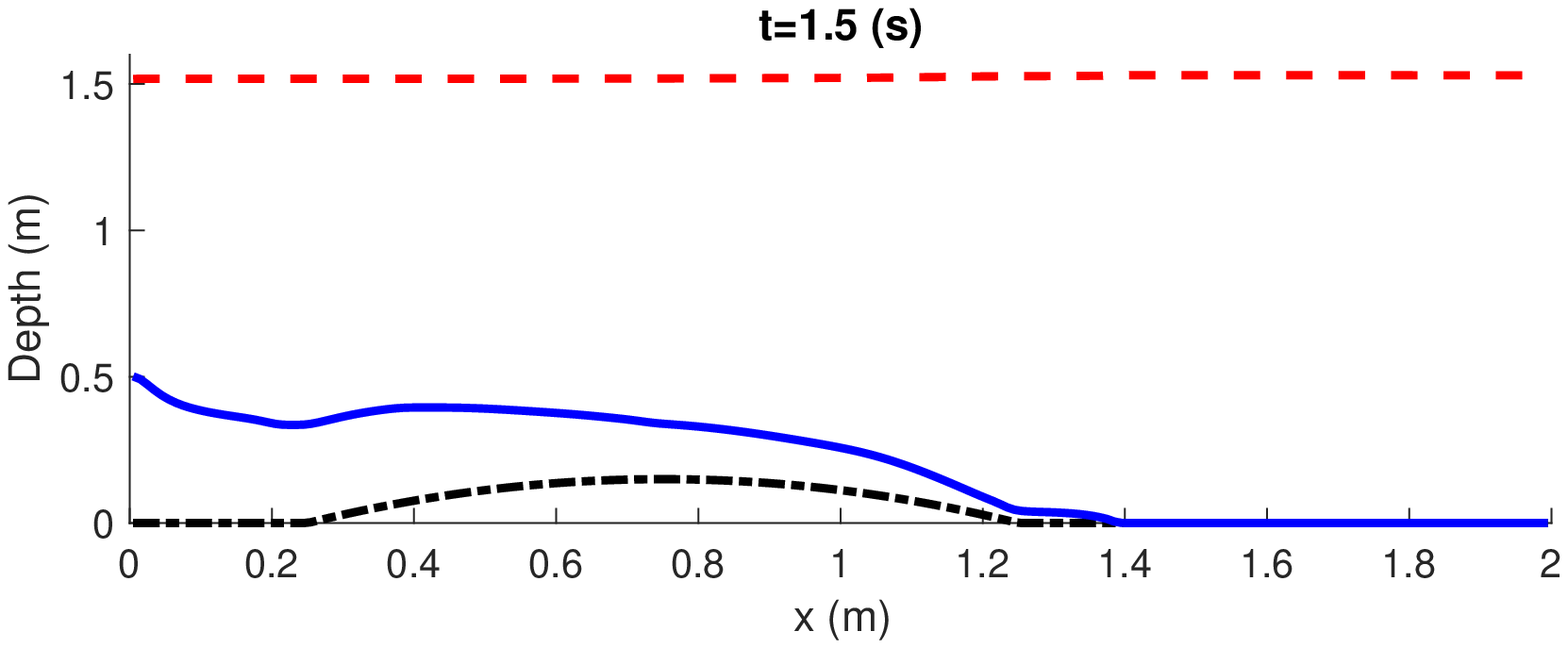} \ \ \ 
\includegraphics[width=0.4\textwidth]{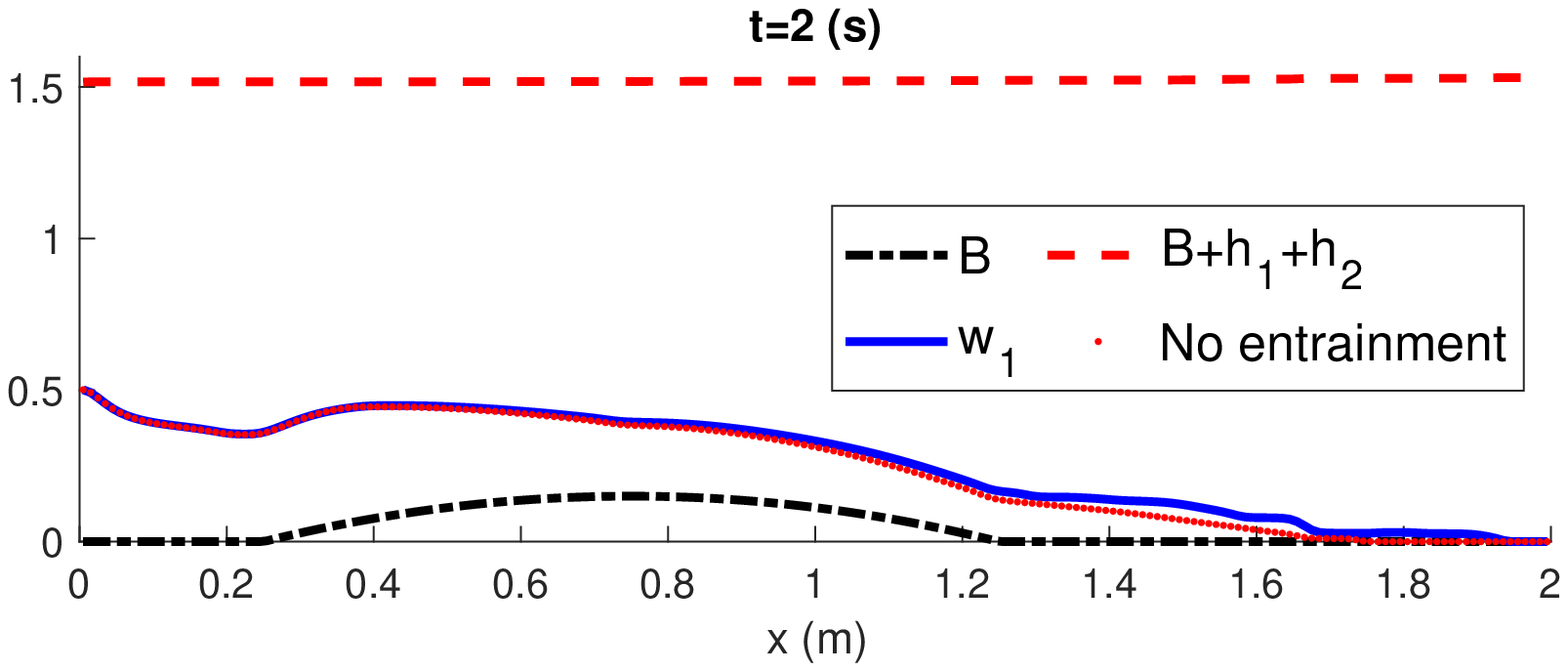} }
\caption{\label{fig:Currents} Time evolution of the solution with initial conditions given by equation \eqref{eq:Currents}. The topography, $w_1$ and $w_2$ are shown at times $t=0,1,1.5,2$. The initial conditions are specified in equation \eqref{eq:Currents}.} 
\end{figure}

Gravity currents produced by lock exchanges have been studied in \cite{adduce2012gravity}. Experimental data and numerical results are compared in \cite{adduce2012gravity} using a shallow-water model with entrainment. It was shown that entrainment helps getting more realistic numerical results. In the numerical test considered in this section, we initiate with a fluid that is composed only of the external layer (lighter fluid). We then impose a discharge of heavier fluid in the internal layer. Specifically, we impose a discharge $Q_{1,\text{left}} = 0.1$, and $w_{1,\text{left}}=0.6, Q_{2,\text{left}}=0, w_{2,\text{left}}=1.5$ at the left boundary. The right boundary is treated as specified at the beginning of Section \ref{sec:Results}. 

The topography $B(x)$ and the geometry $\sigma(x,z)$ are the same from the previous test, and
\begin{equation}
\label{eq:Currents}
 h_1(x,0) = \delta_B, \ \ \ u_1(x,0) = 0, \ \ \ w_2(x,0) = 1.5, \ \ \ u_2(x,0) = 0.
\end{equation}
Here, $r=0.95$, $g=9.81 \text{ m}^2\text{s}^{-1}$, and $n_i = n_b = 0.009 \text{ s m}^{-1/3} $. Following \cite{adduce2012gravity}, the velocity entrainment is given by
\[
V_e = \frac{k G^2}{G^2+5} u_1,
\]
where $G$ is the composite Froude number in equation \eqref{eq:CompositeFroude}, and $k=0.1$. Furthermore, the Manning coefficients for the interface and bottom are $n_i = n_b = 0.009 \text{ s m}^{-1/3} $ respectively.

One of the challenges in this numerical test is that the internal layer has wet-dry states since the depth $h_1$ is initially small. The discharge on the left produces a wet-dry front propagating to the right, inducing a current in the internal layer. Figure \ref{fig:Currents} shows the topography (black dashed line), $w_1$ (solid blue line) and $w_2$ (red dashed line) at times $t=0,1,1.5,2$. One can clearly identify the wet-dry front in the internal layer. In addition to the above variables, the bottom right panel also includes the elevation of the internal layer computed without entrainment. Consistently, one can observe a higher elevation near the front for the computations implemented with entrainment.


\vskip 20 pt
\noindent
{\bf Acknowledgements:} Research was supported in part by grants UNAM-DGAPA-PAPIIT IN113019 \& Conacyt A1-S-17634 (Investigación realizada gracias a los programas UNAM-DGAPA-PAPIIT IN113019 \& Conacyt A1-S-17634 )

\bibliography{References}   

\end{document}